\documentclass[a4paper, 12pt]{amsart}
\usepackage{amssymb, amsmath, amsthm, color, mathrsfs, mathtools}
\usepackage[colorlinks=true, breaklinks=true, linkcolor=black, citecolor=blue, urlcolor=red]{hyperref} 
\usepackage[english]{babel}
\usepackage{tikz-cd}
\usepackage{tikz}
\usetikzlibrary{shapes}

\usepackage{enumitem}

\usepackage{caption}
\usepackage{floatrow}   
\usepackage{subcaption}

\numberwithin{equation}{section}
\newtheorem{theorem}{Theorem}[section]

\newtheorem{lemma}[theorem]{Lemma}
\newtheorem{corollary}[theorem]{Corollary}
\newtheorem{proposition}[theorem]{Proposition}

\newtheorem{notation}[theorem]{Notation}
\newtheorem{example}[theorem]{Example}

\newtheorem{assumption}{Assumption}

\newtheorem{mainthm}{Theorem}

\theoremstyle{definition} 
\newtheorem{definition}[theorem]{Definition}
\newtheorem{remark}[theorem]{Remark}

\newcommand{\act}{\curvearrowright}

\newcommand{\al}{\alpha}

\newcommand{\cC}{\mathcal C}
\newcommand{\fC}{\mathfrak C}

\DeclareMathOperator{\Cone}{Cone}

\DeclareMathOperator{\Ded}{Ded}

\newcommand{\cD}{\mathcal D}
\newcommand{\cE}{\mathcal E}

\newcommand{\cF}{\mathcal F}

\DeclareMathOperator{\Frac}{Frac}

\DeclareMathOperator{\FS}{FS}

\newcommand{\cG}{\mathcal G}

\DeclareMathOperator{\Gr}{Gr}

\newcommand{\ga}{\gamma}
\newcommand{\Ga}{\Gamma}

\newcommand{\cH}{\mathcal H}

\DeclareMathOperator{\Homeo}{Homeo}

\DeclareMathOperator{\id}{id}
\newcommand{\into}{\hookrightarrow}
\newcommand{\cJ}{\mathcal{J}}

\newcommand{\la}{\langle}

\DeclareMathOperator{\Leaf}{Leaf}

\DeclareMathOperator{\Mon}{Mon}

\newcommand{\N}{\mathbf{N}}

\newcommand{\onto}{\twoheadrightarrow}
\newcommand{\ot}{\otimes}
\newcommand{\ov}{\overline}
\newcommand{\cP}{\mathcal P}

\DeclareMathOperator{\prune}{prune}

\newcommand{\fQ}{\mathfrak Q}
\newcommand{\R}{\mathbf{R}}

\newcommand{\fR}{\mathfrak R}
\newcommand{\ra}{\rangle}

\DeclareMathOperator{\Root}{Root}

\DeclareMathOperator{\supp}{supp}

\newcommand{\cT}{\mathcal T}

\DeclareMathOperator{\Ver}{Ver}

\newcommand{\Z}{\mathbf{Z}}

\newcommand{\0}{{\tt0}}
\newcommand{\1}{{\tt1}}

\newcommand{\UncolouredCaret}{
	\begin{tikzpicture}
		\draw[thick] (0,0) -- (-0.5,0.75);
		\draw[thick] (0,0) -- (0.5,0.75);
		\draw[black,fill=black,thick] (0,0) circle [radius=0.06];
		\draw[fill=black] (-0.5,0.75) circle [radius=0.06];
		\draw[fill=black] (0.5,0.75) circle [radius=0.06];
	\end{tikzpicture}  
}
\newcommand{\UncolouredVineA}{
	\begin{tikzpicture}
		\draw[thick] (0,0) -- (-0.5,0.75);
		\draw[thick] (0,0) -- (0.5,0.75);
		\draw[thick] (0.5,0.75) -- (1,1.5);
		\draw[thick] (0.5,0.75) -- (0,1.5);
		\draw[black,fill=black,thick] (0,0) circle [radius=0.06];
		\draw[fill=black] (0.5,0.75) circle [radius=0.06];
		\draw[fill=black] (1,1.5) circle [radius=0.06];
		\draw[fill=black] (0,1.5) circle [radius=0.06];
		\draw[fill=black] (-0.5,0.75) circle [radius=0.06];
	\end{tikzpicture}  
}
\newcommand{\UncolouredVineB}{
	\begin{tikzpicture}
		\draw[thick] (0,0) -- (-0.5,0.75);
		\draw[thick] (0,0) -- (0.5,0.75);
		\draw[thick] (0.5,0.75) -- (1,1.5);
		\draw[thick] (0.5,0.75) -- (0,1.5);
		\draw[thick] (1,1.5) -- (1.5,2.25);
		\draw[thick] (1,1.5) -- (0.5,2.25);
		\draw[black,fill=black,thick] (0,0) circle [radius=0.06];
		\draw[fill=black] (0.5,0.75) circle [radius=0.06];
		\draw[fill=black] (1,1.5) circle [radius=0.06];
		\draw[fill=black] (1.5,2.25) circle [radius=0.06];
		\draw[fill=black] (0.5,2.25) circle [radius=0.06];
		\draw[fill=black] (0,1.5) circle [radius=0.06];
		\draw[fill=black] (-0.5,0.75) circle [radius=0.06];
	\end{tikzpicture}  
}

\newcommand{\ClearyThree}{
\begin{tikzpicture}
\draw[thick] (0+3,0) -- (-0.5+3,0.75);
\draw[thick] (0+3,0) -- (0.5+3,0.75);
\draw[thick] (0.5+3,0.75) -- (1+3,1.5);
\draw[thick] (0.5+3,0.75) -- (0+3,1.5);
\draw[thick] (1+3,1.5) -- (3.5,2.25);
\draw[thick] (1+3,1.5) -- (4.5,2.25);
\draw[black,fill=white,thick] (0+3,0) circle [radius=0.25] node {$b$};
\draw[black,fill=white,thick] (0.5+3,0.75) circle [radius=0.25] node {$b$};
\draw[black,fill=white,thick] (1+3,1.5) circle [radius=0.25] node {$b$};
\draw[fill=black] (3.5,2.25) circle [radius=0.06];
\draw[fill=black] (4.5,2.25) circle [radius=0.06];
\draw[fill=black] (0+3,1.5) circle [radius=0.06];
\draw[fill=black] (-0.5+3,0.75) circle [radius=0.06];
\node at (1.5,0) {$\sim$};
\draw[thick] (0,0) -- (-0.5,0.75);
\draw[thick] (0,0) -- (0.5,0.75);
\draw[thick] (-0.5,0.75) -- (-1,1.5);
\draw[thick] (-0.5,0.75) -- (0,1.5);
\draw[thick] (-1,1.5) -- (-1.5,2.25);
\draw[thick] (-1,1.5) -- (-0.5,2.25);
\draw[black,fill=white,thick] (0,0) circle [radius=0.25] node {$a$};
\draw[black,fill=white,thick] (-0.5,0.75) circle [radius=0.25] node {$a$};
\draw[black,fill=white,thick] (-1,1.5) circle [radius=0.25] node {$a$};
\draw[fill=black] (0,1.5) circle [radius=0.06];
\draw[fill=black] (0.5,0.75) circle [radius=0.06];
\draw[fill=black] (-1.5,2.25) circle [radius=0.06];
\draw[fill=black] (-0.5,2.25) circle [radius=0.06];
\end{tikzpicture}  
}

\newcommand{\ClearyB}{
\begin{tikzpicture}
\draw[thick] (0+3,0) -- (-0.5+3,0.75);
\draw[thick] (0+3,0) -- (0.5+3,0.75);
\draw[thick] (0.5+3,0.75) -- (1+3,1.5);
\draw[thick] (0.5+3,0.75) -- (0+3,1.5);
\draw[black,fill=white,thick] (0+3,0) circle [radius=0.25] node {$b$};
\draw[black,fill=white,thick] (0.5+3,0.75) circle [radius=0.25] node {$b$};
\draw[fill=black] (1+3,1.5) circle [radius=0.06];
\draw[fill=black] (0+3,1.5) circle [radius=0.06];
\draw[fill=black] (-0.5+3,0.75) circle [radius=0.06];
\node at (1.5,0) {$\sim$};
\draw[thick] (0,0) -- (-0.5,0.75);
\draw[thick] (0,0) -- (0.5,0.75);
\draw[thick] (-0.5,0.75) -- (-1,1.5);
\draw[thick] (-0.5,0.75) -- (0,1.5);
\draw[black,fill=white,thick] (0,0) circle [radius=0.25] node {$a$};
\draw[black,fill=white,thick] (-0.5,0.75) circle [radius=0.25] node {$a$};
\draw[fill=black] (-1,1.5) circle [radius=0.06];
\draw[fill=black] (0,1.5) circle [radius=0.06];
\draw[fill=black] (0.5,0.75) circle [radius=0.06];
\end{tikzpicture}  
}

\newcommand{\CompleteABC}{
\begin{tikzpicture}
\draw[fill=black] (0,0) circle [radius=0.05];
\draw[thick] (0,0) -- (-0.5,0.5);
\draw[thick] (0,0) -- (0.5,0.5);
\draw[thick] (-0.5,0.5) -- (-0.75,1);
\draw[thick] (-0.5,0.5) -- (-0.25,1);
\draw[thick] (0.5,0.5) -- (0.75,1);
\draw[thick] (0.5,0.5) -- (0.25,1);
\draw[fill=black] (-0.75,1) circle [radius=0.05];
\draw[fill=black] (-0.25,1) circle [radius=0.05];
\draw[fill=black] (0.75,1) circle [radius=0.05];
\draw[fill=black] (0.25,1) circle [radius=0.05];
\draw[black,fill=white,thick] (0,0) circle [radius=0.25] node {$a$};
\draw[black,fill=white,thick] (-0.5,0.5) circle [radius=0.25] node {$b$};
\draw[black,fill=white,thick] (0.5,0.5) circle [radius=0.25] node {$c$};
\end{tikzpicture}
}

\newcommand{\SkeinPresC}{
\begin{tikzpicture}
\draw[thick] (-1.875,0.875)--(-1.5,1.625);
\draw[thick] (-1.875,0.875)--(-1.5,0.125);
\draw[fill=black] (-1,0.5) circle [radius=0.05];
\draw[thick] (-1,0.5) -- (-1.25,1);
\draw[thick] (-1,0.5) -- (-0.75,1);
\draw[fill=black] (-1.25,1) circle [radius=0.05];
\draw[fill=black] (-0.75,1) circle [radius=0.05];
\draw[black,fill=white,thick] (-1,0.5) circle [radius=0.25] node {$1$};
\node at (-0.25,0.25) {$,$};
\draw[fill=black] (0.5,0.5) circle [radius=0.05];
\draw[thick] (0.5,0.5) -- (0.25,1);
\draw[thick] (0.5,0.5) -- (0.75,1);
\draw[fill=black] (0.25,1) circle [radius=0.05];
\draw[fill=black] (0.75,1) circle [radius=0.05];
\draw[black,fill=white,thick] (0.5,0.5) circle [radius=0.25] node {$2$};
\node at (1.5,0.25) {$,\dots,$};
\draw[fill=black] (2.5,0.5) circle [radius=0.05];
\draw[thick] (2.5,0.5) -- (2.25,1);
\draw[thick] (2.5,0.5) -- (2.75,1);
\draw[fill=black] (2.25,1) circle [radius=0.05];
\draw[fill=black] (2.75,1) circle [radius=0.05];
\draw[black,fill=white,thick] (2.5,0.5) circle [radius=0.25] node {$n$};
\draw[thick] (3.25,1.625) -- (3.25,0.125);
\draw[fill=black] (4.25,0.5) circle [radius=0.05];
\draw[thick] (4.25,0.5) -- (3.75,1);
\draw[thick] (4.25,0.5) -- (4.75,1);
\draw[thick] (4.75,1) -- (4.5,1.5);
\draw[thick] (4.75,1) -- (5,1.5);
\draw[fill=black] (3.75,1) circle [radius=0.05];
\draw[fill=black] (4.5,1.5) circle [radius=0.05];
\draw[fill=black] (5,1.5) circle [radius=0.05];
\draw[black,fill=white,thick] (4.25,0.5) circle [radius=0.25] node {$i$};
\draw[black,fill=white,thick] (4.75,1) circle [radius=0.25] node {$j$};
\node at (5.25,0.5) {$=$};
\draw[fill=black] (6.25,0.5) circle [radius=0.05];
\draw[thick] (6.25,0.5) -- (5.75,1);
\draw[thick] (6.25,0.5) -- (6.75,1);
\draw[thick] (5.75,1) -- (5.5,1.5);
\draw[thick] (5.75,1) -- (6,1.5);
\draw[fill=black] (6.75,1) circle [radius=0.05];
\draw[fill=black] (5.5,1.5) circle [radius=0.05];
\draw[fill=black] (6,1.5) circle [radius=0.05];
\draw[black,fill=white,thick] (6.25,0.5) circle [radius=0.25] node {$j$};
\draw[black,fill=white,thick] (5.75,1) circle [radius=0.25] node {$i$};
\node at (8.25,0.5) {$:~0\leqslant i<j\leqslant n$};
\draw[thick] (10.25,0.875)--(9.875,1.625);
\draw[thick] (10.25,0.875)--(9.875,0.125);
\end{tikzpicture}
}

\newcommand{\SkeinPresA}{
\begin{tikzpicture}
\draw[thick] (-1.875,0.875)--(-1.375,2);
\draw[thick] (-1.875,0.875)--(-1.375,-0.25);
\draw[fill=black] (-1,0.5) circle [radius=0.05];
\draw[thick] (-1,0.5) -- (-1.25,1);
\draw[thick] (-1,0.5) -- (-0.75,1);
\draw[fill=black] (-1.25,1) circle [radius=0.05];
\draw[fill=black] (-0.75,1) circle [radius=0.05];
\draw[black,fill=white,thick] (-1,0.5) circle [radius=0.25] node {$a$};
\node at (-0.25,0.25) {$,$};
\draw[fill=black] (0.5,0.5) circle [radius=0.05];
\draw[thick] (0.5,0.5) -- (0.25,1);
\draw[thick] (0.5,0.5) -- (0.75,1);
\draw[fill=black] (0.25,1) circle [radius=0.05];
\draw[fill=black] (0.75,1) circle [radius=0.05];
\draw[black,fill=white,thick] (0.5,0.5) circle [radius=0.25] node {$b$};
\draw[thick] (1.25,2) -- (1.25,-0.25);
\draw[fill=black] (3,0) circle [radius=0.05];
\draw[thick] (3,0) -- (2.5,0.5);
\draw[thick] (3,0) -- (3.5,0.5);
\draw[thick] (2.5,0.5) -- (1.75,2);
\draw[thick] (2.5,0.5) -- (3.25,2);
\draw[thick] (1.75,2) -- (3.25,2);
\draw[fill=black] (1.75,2) circle [radius=0.05];
\draw[fill=black] (3.25,2) circle [radius=0.05];
\draw[fill=black] (2.5,0.5) circle [radius=0.05];
\draw[fill=black] (3.5,0.5) circle [radius=0.05];
\draw[black,fill=white,thick] (2.5,1.5) node {$t(a)$};
\draw[black,fill=white,thick] (3,0) circle [radius=0.25] node {$a$};
\node at (4.75,0.25) {$=$};
\draw[thick] (6,0) -- (5.5,0.5);
\draw[thick] (6,0) -- (6.5,0.5);
\draw[thick] (6.5,0.5) -- (6,1);
\draw[thick] (6.5,0.5) -- (6.75,0.75);
\node[rotate=45] at (7,1) {$\cdots$};
\draw[thick] (7.5,1.5) -- (7.25,1.25);
\draw[thick] (7.5,1.5) -- (7,2);
\draw[thick] (7.5,1.5) -- (8,2);
\draw[fill=black] (5.5,0.5) circle [radius=0.05];
\draw[fill=black] (6,1) circle [radius=0.05];
\draw[fill=black] (7,2) circle [radius=0.05];
\draw[fill=black] (8,2) circle [radius=0.05];
\draw[black,fill=white,thick] (6,0) circle [radius=0.25] node {$b$};
\draw[black,fill=white,thick] (6.5,0.5) circle [radius=0.25] node {$b$};
\draw[black,fill=white,thick] (7.5,1.5) circle [radius=0.25] node {$b$};
\draw[thick] (8.75,0.875)--(8.25,2);
\draw[thick] (8.75,0.875)--(8.25,-0.25);
\end{tikzpicture}
}

\newcommand{\SkeinPresB}{
\begin{tikzpicture}
\draw[thick] (-1.875,0.875)--(-1.375,2);
\draw[thick] (-1.875,0.875)--(-1.375,-0.25);
\draw[fill=black] (-1,0.5) circle [radius=0.05];
\draw[thick] (-1,0.5) -- (-1.25,1);
\draw[thick] (-1,0.5) -- (-0.75,1);
\draw[fill=black] (-1.25,1) circle [radius=0.05];
\draw[fill=black] (-0.75,1) circle [radius=0.05];
\draw[black,fill=white,thick] (-1,0.5) circle [radius=0.25] node {$a$};
\node at (-0.25,0.25) {$,$};
\draw[fill=black] (0.5,0.5) circle [radius=0.05];
\draw[thick] (0.5,0.5) -- (0.25,1);
\draw[thick] (0.5,0.5) -- (0.75,1);
\draw[fill=black] (0.25,1) circle [radius=0.05];
\draw[fill=black] (0.75,1) circle [radius=0.05];
\draw[black,fill=white,thick] (0.5,0.5) circle [radius=0.25] node {$b$};
\draw[thick] (1.25,2) -- (1.25,-0.25);
\draw[fill=black] (3,0) circle [radius=0.05];
\draw[thick] (3,0) -- (2.5,0.5);
\draw[thick] (3,0) -- (3.5,0.5);
\draw[thick] (2.5,0.5) -- (1.75,2);
\draw[thick] (2.5,0.5) -- (3.25,2);
\draw[thick] (1.75,2) -- (3.25,2);
\draw[thick] (3.5,0.5) -- (3.75,1);
\draw[thick] (3.5,0.5) -- (3.25,1);
\draw[fill=black] (1.75,2) circle [radius=0.05];
\draw[fill=black] (3.25,2) circle [radius=0.05];
\draw[fill=black] (2.5,0.5) circle [radius=0.05];
\draw[fill=black] (3.75,1) circle [radius=0.05];
\draw[fill=black] (3.25,1) circle [radius=0.05];
\draw[black,fill=white,thick] (2.5,1.5) node {$t(a)$};
\draw[black,fill=white,thick] (3,0) circle [radius=0.25] node {$a$};
\draw[black,fill=white,thick] (3.5,0.5) circle [radius=0.25] node {$a$};
\node at (4.75,0.25) {$=$};
\draw[thick] (6,0) -- (5.5,0.5);
\draw[thick] (6,0) -- (6.5,0.5);
\draw[thick] (6.5,0.5) -- (6,1);
\draw[thick] (6.5,0.5) -- (6.75,0.75);
\node[rotate=45] at (7,1) {$\cdots$};
\draw[thick] (7.5,1.5) -- (7.25,1.25);
\draw[thick] (7.5,1.5) -- (7,2);
\draw[thick] (7.5,1.5) -- (8,2);
\draw[fill=black] (5.5,0.5) circle [radius=0.05];
\draw[fill=black] (6,1) circle [radius=0.05];
\draw[fill=black] (7,2) circle [radius=0.05];
\draw[fill=black] (8,2) circle [radius=0.05];
\draw[black,fill=white,thick] (6,0) circle [radius=0.25] node {$b$};
\draw[black,fill=white,thick] (6.5,0.5) circle [radius=0.25] node {$b$};
\draw[black,fill=white,thick] (7.5,1.5) circle [radius=0.25] node {$b$};
\draw[thick] (8.75,0.875)--(8.25,2);
\draw[thick] (8.75,0.875)--(8.25,-0.25);
\end{tikzpicture}
}

\newcommand{\FigA}{
\begin{tikzpicture}
\draw[fill=black] (0,0) circle [radius=0.05];
\draw[thick] (0,0) -- (-0.5,0.5);
\draw[thick] (0,0) -- (0.5,0.5);
\draw[fill=black] (-0.5,0.5) circle [radius=0.05];
\draw[thick] (-0.5,0.5) -- (-0.75,1);
\draw[thick] (-0.5,0.5) -- (-0.25,1);
\draw[thick] (0.5,0.5) -- (0.25,1);
\draw[thick] (0.5,0.5) -- (0.75,1);
\draw[fill=black] (-0.75,1) circle [radius=0.05];
\draw[fill=black] (0.75,1) circle [radius=0.05];
\draw[fill=black] (-0.25,1) circle [radius=0.05];
\draw[fill=black] (0.25,1) circle [radius=0.05];
\draw[fill=black] (0.5,0.5) circle [radius=0.05];
\end{tikzpicture}
}

\newcommand{\FigB}{
\begin{tikzpicture}
\draw[fill=black] (0,0) circle [radius=0.05];
\draw[thick] (0,0) -- (-0.5,0.5);
\draw[thick] (0,0) -- (0.5,0.5);
\draw[fill=black] (-0.5,0.5) circle [radius=0.05];
\draw[thick] (-0.5,0.5) -- (-1,1);
\draw[thick] (-0.5,0.5) -- (0,1);
\draw[fill=black] (-1,1) circle [radius=0.05];
\draw[fill=black] (0,1) circle [radius=0.05];
\draw[fill=black] (0.5,0.5) circle [radius=0.05];
\draw[fill=black] (2,0) circle [radius=0.05];
\draw[thick] (2,0) -- (1.5,0.5);
\draw[thick] (2,0) -- (2.5,0.5);
\draw[fill=black] (1.5,0.5) circle [radius=0.05];
\draw[fill=black] (2.5,0.5) circle [radius=0.05];
\end{tikzpicture}
}

\newcommand{\FigC}{
\begin{tikzpicture}
\draw[fill=black] (0,0) circle [radius=0.05];
\draw[thick] (0,0) -- (-0.5,0.5);
\draw[thick] (0,0) -- (0.5,0.5);
\draw[fill=black] (-0.5,0.5) circle [radius=0.05];
\draw[thick] (-0.5,0.5) -- (-1,1);
\draw[thick] (-0.5,0.5) -- (0,1);
\draw[fill=black] (0,1) circle [radius=0.05];
\draw[thick] (0,1) -- (-0.5,1.5);
\draw[thick] (0,1) -- (0.5,1.5);
\draw[fill=black] (-0.5,1.5) circle [radius=0.05];
\draw[fill=black] (0.5,1.5) circle [radius=0.05];
\draw[fill=black] (-1,1) circle [radius=0.05];
\draw[fill=black] (0.5,0.5) circle [radius=0.05];
\end{tikzpicture}
}

\newcommand{\FigD}{
\begin{tikzpicture}
\draw[fill=black] (0,0) circle [radius=0.05];
\draw[thick] (0,0) -- (-0.5,0.5);
\draw[thick] (0,0) -- (0.5,0.5);
\draw[fill=black] (-0.5,0.5) circle [radius=0.05];
\draw[thick] (-0.5,0.5) -- (-1,1);
\draw[thick] (-0.5,0.5) -- (0,1);
\draw[fill=black] (-1,1) circle [radius=0.05];
\draw[fill=black] (0,1) circle [radius=0.05];
\draw[fill=black] (0.5,0.5) circle [radius=0.05];
\draw[fill=black] (2,0) circle [radius=0.05];
\draw[thick] (2,0) -- (1.5,0.5);
\draw[thick] (2,0) -- (2.5,0.5);
\draw[fill=black] (1.5,0.5) circle [radius=0.05];
\draw[fill=black] (2.5,0.5) circle [radius=0.05];
\draw[black,fill=white,thick] (0,0) circle [radius=0.25] node {$a$};
\draw[black,fill=white,thick] (-0.5,0.5) circle [radius=0.25] node {$b$};
\draw[black,fill=white,thick] (2,0) circle [radius=0.25] node {$c$};
\end{tikzpicture}
}

\newcommand{\Composition}{
\begin{tikzpicture}
\draw[fill=black] (0,0) circle [radius=0.05];
\draw[thick] (0,0) -- (-0.5,0.5);
\draw[thick] (0,0) -- (0.5,0.5);
\draw[fill=black] (-0.5,0.5) circle [radius=0.05];
\draw[thick] (1,-0.2) -- (1,0.5);
\draw[fill=black] (1,-0.2) circle [radius=0.05];
\draw[fill=black] (1,0.5) circle [radius=0.05];
\draw[fill=black] (0.5,0.5) circle [radius=0.05];
\draw[black,fill=white,thick] (0,0) circle [radius=0.25] node {$a$};
\node at (1.75,0.15) {$\circ$};
\draw[thick] (2.5,-0.2) -- (2.5,0.5);
\draw[thick] (3,-0.2) -- (3,0.5);
\draw[fill=black] (2.5,-0.2) circle [radius=0.05];
\draw[fill=black] (2.5,0.5) circle [radius=0.05];
\draw[fill=black] (3,-0.2) circle [radius=0.05];
\draw[fill=black] (3,0.5) circle [radius=0.05];
\draw[fill=black] (4,0) circle [radius=0.05];
\draw[thick] (4,0) -- (3.5,0.5);
\draw[thick] (4,0) -- (4.5,0.5);
\draw[fill=black] (3.5,0.5) circle [radius=0.05];
\draw[fill=black] (4.5,0.5) circle [radius=0.05];
\draw[black,fill=white,thick] (4,0) circle [radius=0.25] node {$b$};
\node at (5,0.1) {$=$};
\draw[fill=black] (6,0) circle [radius=0.05];
\draw[thick] (6,0) -- (5.75,0.5);
\draw[thick] (6,0) -- (6.25,0.5);
\draw[thick] (7,-0.2) -- (7,0.5);
\draw[fill=black] (7,-0.2) circle [radius=0.05];
\draw[fill=black] (5.75,0.5) circle [radius=0.05];
\draw[fill=black] (6.25,0.5) circle [radius=0.05];
\draw[black,fill=white,thick] (6,0) circle [radius=0.25] node {$a$};
\draw[thick] (7,0.5) -- (6.75,1);
\draw[thick] (7,0.5) -- (7.25,1);
\draw[thick] (5.75,0.5) -- (5.75,1);
\draw[thick] (6.25,0.5) -- (6.25,1);
\draw[fill=black] (5.75,1) circle [radius=0.05];
\draw[fill=black] (6.25,1) circle [radius=0.05];
\draw[fill=black] (6.75,1) circle [radius=0.05];
\draw[fill=black] (7.25,1) circle [radius=0.05];
\draw[black,fill=white,thick] (7,0.5) circle [radius=0.25] node {$b$};
\node at (8,0.1) {$=$};
\draw[fill=black] (9,0) circle [radius=0.05];
\draw[thick] (9,0) -- (8.75,0.5);
\draw[thick] (9,0) -- (9.25,0.5);
\draw[fill=black] (8.75,0.5) circle [radius=0.05];
\draw[fill=black] (9.25,0.5) circle [radius=0.05];
\draw[black,fill=white,thick] (9,0) circle [radius=0.25] node {$a$};
\draw[thick] (10,0) -- (9.75,0.5);
\draw[thick] (10,0) -- (10.25,0.5);
\draw[fill=black] (9.75,0.5) circle [radius=0.05];
\draw[fill=black] (10.25,0.5) circle [radius=0.05];
\draw[black,fill=white,thick] (10,0) circle [radius=0.25] node {$b$};
\end{tikzpicture}
}

\newcommand{\ACaret}{
\begin{tikzpicture}
\draw[fill=black] (0,0) circle [radius=0.05];
\draw[thick] (0,0) -- (-0.5,0.5);
\draw[thick] (0,0) -- (0.5,0.5);
\draw[fill=black] (-0.5,0.5) circle [radius=0.05];
\draw[fill=black] (0.5,0.5) circle [radius=0.05];
\draw[black,fill=white,thick] (0,0) circle [radius=0.25] node {$a$};
\end{tikzpicture}
}

\newcommand{\TensorProd}{
\begin{tikzpicture}
\draw[fill=black] (0,0) circle [radius=0.05];
\draw[thick] (0,0) -- (-0.5,0.5);
\draw[thick] (0,0) -- (0.5,0.5);
\draw[fill=black] (-0.5,0.5) circle [radius=0.05];
\draw[thick] (1,-0.2) -- (1,0.5);
\draw[fill=black] (1,-0.2) circle [radius=0.05];
\draw[fill=black] (1,0.5) circle [radius=0.05];
\draw[fill=black] (0.5,0.5) circle [radius=0.05];
\draw[black,fill=white,thick] (0,0) circle [radius=0.25] node {$a$};
\node at (1.75,0.15) {$\ot$};
\draw[thick] (2.5,-0.2) -- (2.5,0.5);
\draw[thick] (3,-0.2) -- (3,0.5);
\draw[fill=black] (2.5,-0.2) circle [radius=0.05];
\draw[fill=black] (2.5,0.5) circle [radius=0.05];
\draw[fill=black] (3,-0.2) circle [radius=0.05];
\draw[fill=black] (3,0.5) circle [radius=0.05];
\draw[fill=black] (4,0) circle [radius=0.05];
\draw[thick] (4,0) -- (3.5,0.5);
\draw[thick] (4,0) -- (4.5,0.5);
\draw[fill=black] (3.5,0.5) circle [radius=0.05];
\draw[fill=black] (4.5,0.5) circle [radius=0.05];
\draw[black,fill=white,thick] (4,0) circle [radius=0.25] node {$b$};
\node at (5,0.1) {$=$};
\draw[thick] (6,0) -- (5.5,0.5);
\draw[thick] (6,0) -- (6.5,0.5);
\draw[fill=black] (5.5,0.5) circle [radius=0.05];
\draw[fill=black] (6.5,0.5) circle [radius=0.05];
\draw[thick] (7,-0.2) -- (7,0.5);
\draw[thick] (7.5,-0.2) -- (7.5,0.5);
\draw[thick] (8,-0.2) -- (8,0.5);
\draw[fill=black] (7,-0.2) circle [radius=0.05];
\draw[fill=black] (7,0.5) circle [radius=0.05];
\draw[fill=black] (7.5,-0.2) circle [radius=0.05];
\draw[fill=black] (7.5,0.5) circle [radius=0.05];
\draw[fill=black] (8,-0.2) circle [radius=0.05];
\draw[fill=black] (8,0.5) circle [radius=0.05];
\draw[black,fill=white,thick] (6,0) circle [radius=0.25] node {$a$};
\draw[thick] (9,0) -- (8.5,0.5);
\draw[thick] (9,0) -- (9.5,0.5);
\draw[fill=black] (8.5,0.5) circle [radius=0.05];
\draw[fill=black] (9.5,0.5) circle [radius=0.05];
\draw[black,fill=white,thick] (9,0) circle [radius=0.25] node {$b$};
\end{tikzpicture}
}

\newcommand{\SkeinRelation}{
\begin{tikzpicture}
\draw[fill=black] (0,0) circle [radius=0.05];
\draw[thick] (0,0) -- (-0.5,0.5);
\draw[thick] (0,0) -- (0.5,0.5);
\draw[thick] (-0.5,0.5) -- (-0.75,1);
\draw[thick] (-0.5,0.5) -- (-0.25,1);
\draw[thick] (0.5,0.5) -- (0.75,1);
\draw[thick] (0.5,0.5) -- (0.25,1);
\draw[fill=black] (-0.75,1) circle [radius=0.05];
\draw[fill=black] (-0.25,1) circle [radius=0.05];
\draw[fill=black] (0.75,1) circle [radius=0.05];
\draw[fill=black] (0.25,1) circle [radius=0.05];
\draw[black,fill=white,thick] (0,0) circle [radius=0.25] node {$a$};
\draw[black,fill=white,thick] (-0.5,0.5) circle [radius=0.25] node {$a$};
\draw[black,fill=white,thick] (0.5,0.5) circle [radius=0.25] node {$a$};
\node at (1.5,0) {$\sim$};
\draw[thick] (3,0) -- (2.5,0.5);
\draw[thick] (3,0) -- (3.5,0.5);
\draw[thick] (3.5,0.5) -- (3,1);
\draw[thick] (3.5,0.5) -- (4,1);
\draw[thick] (4,1) -- (3.5,1.5);
\draw[thick] (4,1) -- (4.5,1.5);
\draw[fill=black] (2.5,0.5) circle [radius=0.05];
\draw[fill=black] (3,1) circle [radius=0.05];
\draw[fill=black] (3.5,1.5) circle [radius=0.05];
\draw[fill=black] (4.5,1.5) circle [radius=0.05];
\draw[black,fill=white,thick] (3,0) circle [radius=0.25] node {$b$};
\draw[black,fill=white,thick] (3.5,0.5) circle [radius=0.25] node {$b$};
\draw[black,fill=white,thick] (4,1) circle [radius=0.25] node {$b$};
\end{tikzpicture}
}

\newcommand{\Cabelling}{
\begin{tikzpicture}
\node at (-2.5,-0.75) {$(12)\circ(Y_a\ot I\ot I)=$};
\node at (7.5,-0.75) {$=(I\ot Y_a\ot I)\circ(123)$};
\draw[thick] (1,-1) -- (0,0);
\draw[thick] (0,-1) -- (1,0);
\draw[thick] (1.75,-1) -- (1.75,0);
\draw[fill=black] (0,-1) circle [radius=0.05];
\draw[fill=black] (1,-1) circle [radius=0.05];
\draw[fill=black] (1.75,-1) circle [radius=0.05];
\draw[fill=black] (0,0) circle [radius=0.05];
\draw[thick] (0,0) -- (-0.5,0.5);
\draw[thick] (0,0) -- (0.5,0.5);
\draw[fill=black] (-0.5,0.5) circle [radius=0.05];
\draw[thick] (1,0) -- (1,0.5);
\draw[thick] (1.75,0) -- (1.75,0.5);
\draw[fill=black] (1,0) circle [radius=0.05];
\draw[fill=black] (1,0.5) circle [radius=0.05];
\draw[fill=black] (1.75,0) circle [radius=0.05];
\draw[fill=black] (1.75,0.5) circle [radius=0.05];
\draw[fill=black] (0.5,0.5) circle [radius=0.05];
\draw[black,fill=white,thick] (0,0) circle [radius=0.25] node {$a$};
\node at (2.5,-0.75) {$=$};
\draw[thick] (3.25,-1) -- (3.25,-0.5);
\draw[thick] (4.25,-1) -- (3.75,-0.5);
\draw[thick] (4.25,-1) -- (4.75,-0.5);
\draw[thick] (5.25,-1) -- (5.25,-0.5);
\draw[fill=black] (3.25,-1) circle [radius=0.05];
\draw[fill=black] (3.25,-0.5) circle [radius=0.05];
\draw[fill=black] (3.75,-0.5) circle [radius=0.05];
\draw[fill=black] (4.75,-0.5) circle [radius=0.05];
\draw[fill=black] (5.25,-1) circle [radius=0.05];
\draw[fill=black] (5.25,-0.5) circle [radius=0.05];
\draw[thick] (3.75,-0.5) -- (3.25,0.5);
\draw[thick] (4.75,-0.5) -- (4.25,0.5);
\draw[thick] (3.25,-0.5) -- (4.75,0.5);
\draw[thick] (5.25,-0.5) -- (5.25,0.5);
\draw[fill=black] (3.25,0.5) circle [radius=0.05];
\draw[fill=black] (4.25,0.5) circle [radius=0.05];
\draw[fill=black] (4.75,0.5) circle [radius=0.05];
\draw[fill=black] (5.25,0.5) circle [radius=0.05];
\draw[black,fill=white,thick] (4.25,-1) circle [radius=0.25] node {$a$};
\end{tikzpicture}
}

\newcommand{\Cell}{
\begin{tikzpicture}
\draw[fill=black] (0,0) circle [radius=0.05];
\draw[thick] (0,0) -- (-0.5,0.5);
\draw[thick] (0,0) -- (0.5,0.5);
\draw[fill=black] (-0.5,0.5) circle [radius=0.05];
\draw[thick] (-0.5,0.5) -- (-1,1);
\draw[thick] (-0.5,0.5) -- (0,1);
\draw[fill=black] (-1,1) circle [radius=0.05];
\draw[fill=black] (0,1) circle [radius=0.05];
\draw[fill=black] (0.5,0.5) circle [radius=0.05];
\end{tikzpicture}
}

\newcommand{\InfiniteQuasiRightVine}{
\begin{tikzpicture}
\draw[fill=black] (0,0) circle [radius=0.05];
\draw[thick] (0,0) -- (-0.5,0.5);
\draw[thick] (0,0) -- (1,1);
\draw[fill=black] (-0.5,0.5) circle [radius=0.05];
\draw[thick] (-0.5,0.5) -- (-1,1);
\draw[thick] (-0.5,0.5) -- (0,1);
\draw[fill=black] (-1,1) circle [radius=0.05];
\draw[fill=black] (0,1) circle [radius=0.05];
\draw[fill=black] (1,1) circle [radius=0.05];
\draw[thick] (1,1) -- (0.5,1.5);
\draw[thick] (1,1) -- (2,2);
\draw[fill=black] (0.5,1.5) circle [radius=0.05];
\draw[thick] (0.5,1.5) -- (0,2);
\draw[thick] (0.5,1.5) -- (1,2);
\draw[fill=black] (0,2) circle [radius=0.05];
\draw[fill=black] (1,2) circle [radius=0.05];
\draw[fill=black] (2,2) circle [radius=0.05];
\draw[thick] (2,2) -- (1.5,2.5);
\draw[thick] (2,2) -- (3,3);
\draw[fill=black] (1.5,2.5) circle [radius=0.05];
\draw[thick] (1.5,2.5) -- (1,3);
\draw[thick] (1.5,2.5) -- (2,3);
\draw[fill=black] (1,3) circle [radius=0.05];
\draw[fill=black] (2,3) circle [radius=0.05];
\node[rotate=45] at (3.25,3.25) {$\cdots$};
\end{tikzpicture}
}

\newcommand{\generators}{
\begin{tikzpicture}
\draw[fill=black] (0,0) circle [radius=0.05];
\draw[thick] (0,0) -- (-0.25,0.5);
\draw[thick] (0,0) -- (0.25,0.5);
\draw[fill=black] (-0.25,0.5) circle [radius=0.05];
\draw[fill=black] (0.25,0.5) circle [radius=0.05];
\draw[black,fill=white,thick] (0,0) circle [radius=0.25] node {$b$};
\draw[fill=black] (0,1) circle [radius=0.05];
\draw[thick] (0,1) -- (-0.25,0.5);
\draw[thick] (0,1) -- (0.25,0.5);
\draw[black,fill=white,thick] (0,1) circle [radius=0.25] node {$a$};
\node at (1.5,0.5) {and};
\draw[fill=black] (2.75,0.5) circle [radius=0.05];
\draw[fill=black] (3.5,0.5) circle [radius=0.05];
\draw[fill=black] (4.25,0.5) circle [radius=0.05];
\draw[fill=black] (5,0.5) circle [radius=0.05];
\draw[thick] (2.75,0.5) -- (3.875,-1);
\draw[thick] (5,0.5) -- (3.875,-1);
\draw[thick] (3.5,0.5) -- (4.25,-0.5);
\draw[thick] (4.25,0.5) -- (3.875,0);
\draw[black,fill=white,thick] (3.875,-1) circle [radius=0.25] node {$a$};
\draw[black,fill=white,thick] (4.25,-0.5) circle [radius=0.25] node {$a$};
\draw[black,fill=white,thick] (3.875,0) circle [radius=0.25] node {$b$};
\draw[thick] (5,0.5) -- (3.875,2);
\draw[thick] (4.25,0.5) -- (4.625,1);
\draw[thick] (3.5,0.5) -- (4.25,1.5);
\draw[thick] (2.75,0.5) -- (3.875,2);
\draw[black,fill=white,thick] (3.875,2) circle [radius=0.25] node {$a$};
\draw[black,fill=white,thick] (4.625,1) circle [radius=0.25] node {$a$};
\draw[black,fill=white,thick] (4.25,1.5) circle [radius=0.25] node {$a$};
\node at (5.25,0.385) {.};
\end{tikzpicture}
}

\begin{document}

\title[Reconstruction for simple FS groups]{McCleary--Rubin reconstruction for simple forest-skein groups
}

\thanks{
AB was supported by the Australian Research Council Grant DP200100067.\\
RS was supported by an Australian Government Research Training Program (RTP) Scholarship.}
\author{Arnaud Brothier and Ryan Seelig}
\address{Arnaud Brothier, University of Trieste, Department of Mathematics, via Valerio 12/1, 34127, Trieste, Italy and
School of Mathematics and Statistics, University of New South Wales, Sydney NSW 2052, Australia}
\email{arnaud.brothier@gmail.com\endgraf
\url{https://sites.google.com/site/arnaudbrothier/}}
\address{Ryan Seelig, School of Mathematics and Statistics, University of New South Wales, Sydney NSW 2052, Australia}
\email{ryanseelig1997@gmail.com\endgraf
\url{https://sites.google.com/view/ryanseelig/}}

\begin{abstract} 
A simple Ore forest-skein category produces three infinite groups analogous to Richard Thompson's groups $F,T,V$. 
We prove that reconstruction theorems of McCleary and Rubin apply to them: each of these groups encodes a canonical action by homeomorphisms. This provides powerful invariants that we use to distinguish infinitely many explicit simple groups that are finitely presented (of type $\tt F_\infty$).
\end{abstract}

\maketitle

\section*{Introduction}

This article aims to provide invariants to distinguish infinite simple groups. The simple groups considered in this article are built using the forest-skein framework. \textit{No prior knowledge on this framework is necessary for reading this article.}

\subsection*{Forest-skein groups} Inspired by deep connections uncovered by Vaughan Jones between Richard Thompson's group $F,T,V$, subfactors, and conformal field theory, the first author has introduced \emph{forest-skein (FS) groups} which are a class of ``Thompson-like'' groups which mix features of $F,T,V$ with Jones' planar algebra; \cite{Jones17,Brothier20,Brothier22}. 

To any set $S$ of \emph{colours} and set $R$ of \emph{skein relations} over $S$ (pair of binary trees with same number of leaves having interior vertices labelled by $S$) one may consider the FS category $\cF$ with \emph{skein presentation} $\la S|R\ra$. This consists of all binary forests with interior vertices coloured by $S$ modulo the skein relations in $R$. 
When $\cF$ admits a \emph{calculus of fractions} one can build groups $G^F\subset G^T\subset G^V$ made of tree-pair diagrams from $\cF$. They are similar to $F\subset T\subset V$, the latter corresponding to skein presentation $\la\ast|\varnothing\ra$ consisting of a single colour $\ast$ and no skein relations. 
We call these the \emph{FS groups} of $\cF$. The force of the FS formalism is the ability to produce interesting groups using very few data.
For instance, finitely presented infinite simple groups with unforseen properties were produced using a skein presentation with two colours and one skein relation; \cite{Brothier-Seelig26}.
Large families of explicit examples of FS groups have been studied in \cite{Brothier22,Brothier23,Brothier-Seelig24,Brothier-Seelig26}.

\subsection*{Simple groups and reconstruction theorems} 
Thompson's groups and their relatives are of great interest to geometric group theory as they have sourced many of the first examples of groups witnessing novel properties; see \cite{Cannon-Floyd-Parry96}. Notably, Thompson's groups $T$ and $V$ were the first examples of finitely presented infinite simple groups. 
A full classification of infinite simple groups remains far out of reach, and we are still today discovering which properties a simple group can satisfy.
It is a challenging problem to distinguish simple groups up to isomorphism, however there do exist powerful dynamical invariants.
The study of how much a group of symmetries remembers about the space it acts on goes back to the $1950$'s and arose originally in the context of reconstructing manifolds from their homeomorphism groups; see \cite{McCleary-Rubin05} for an historical account. More recently, both McCleary and Rubin have deduced mild conditions for a group action $\Gamma\act Z$ on some structure $Z$ (e.g.,~topological space, ordered space) to be \emph{rigid}, meaning that $\Gamma\act Z$ can be reconstructed using \emph{only} the algebraic data of $\Gamma$; 
\cite{McCleary78,Rubin89,McCleary-Rubin05}. Hence ``the'' rigid action of a group is an invariant. Brin and later Bleak and Lanoue used rigid actions to distinguish Brin's higher dimension Thompson groups; \cite{Brin04,Bleak-Lanoue10}.

\subsection*{FS technology produces simple groups}
Let $\cF$ be an Ore FS category (i.e.,~it is left-cancellative and have common right-multiples yielding a calculus of fractions) with FS groups $G^F\subset G^T\subset G^V$.
From $\cF$ the \emph{canonical group action} $\al:G^V\act\fC$ can be built functorially where $\fC$ is often a Cantor space.
We previously showed that faithfulness of $\al:G^V\act\fC$ is equivalent to simplicity of one (resp.~all) of the derived subgroups $D(G^V)$, $D(G^T)$, and $D(G^0)$, where $G^0\lhd G^F$ is the normal subgroup of $g\in\fC$ acting trivially near the extremities of $\fC$ (in previous articles $G^0$ has been denoted $K$); \cite{Brothier-Seelig24}.
If $\cF$ is moreover cancellative (on both sides), then $\al:G^V\act\fC$ is faithful if and only if $\cF$ satisfies a notion of \emph{simplicity}: adding a ``new'' skein relation destroys the calculus of fractions (see section \ref{sec:simplicity}). 
Every FS group with a faithful canonical group action can be realised as the FS group of a (possibly different) simple FS category as explained in \cite[Sec.~4.2]{Brothier-Seelig24}.
The simple groups $D(G^0),D(G^T),D(G^V)$ contains copies of $D(F),T,V$, respectively. This implies various properties of geometric, topological, and analytical natures; see \cite[Sec.~3.3]{Brothier22}. Note that $D(G^0)$ is never finitely generated while $D(G^T),D(G^V)$ are often finitely presented. 

\subsection*{McCleary--Rubin rigidity of FS groups} 
Our first main result is when the canonical group action $\al:G^V\act\fC$ is faithful, it is the rigid action for $G^V$ in the sense of the reconstruction theorems of McCleary and Rubin. 
Hence, not only is $G^V\act\fC$ canonical for the \emph{category} $\cF$, but it is canonical for the \emph{group} $G^V$. 
The restrictions of $G^V\act\fC$ to $G^T$ and $G^F$ are not rigid. This is easily rectified by forming certain obvious subquotients of $\fC$ that we denote $\fR^F$ and $\fR^T$.
For notational consistency we set $\fR^V:=\fC$. 

\begin{mainthm}[McCleary--Rubin rigidity of FS groups;~Theorem~\ref{theo:rigid-FS}]
Let $\cF$ be an Ore FS category with faithful canonical action and FS groups $G^F$, $G^T$, $G^V$. Then, $D(G^X)\act\fR^X$ (and hence $G^X\act\fR^X$) is rigid for $X=F,T,V$. When $\cF$ is countable these rigid spaces are homeomorphic to the real unit interval, the circle, and the Cantor space, respectively.
In particular, $G^X$ ``remembers'' its type $X$ in that if $G^X\simeq H^Y$ for some other FS group $H^Y$, then $X=Y$. 
\end{mainthm}

The hypothesis of countability cannot be removed (see item (5) of remark \ref{rem:rigid-non-homeo}).
For the $F$ and $T$-case we have stronger statements showing that the group remembers an order and a cyclic order, respectively (see section \ref{sec:rigid-FS} for details).

We use the rigid actions to deduce invariants to distinguish explicit simple groups.

\subsection*{Examples and classification} For a natural number $n\geqslant1$ and a non-trivial uncoloured tree $t$ define the FS categories:

\begin{align*}\label{eq:skein-presC}
\tag{$H$}\cH_n:=\FS\vcenter{\hbox{\SkeinPresC}};
\end{align*}
\begin{align*}\label{eq:skein-presA}
\tag{$G$}\cG_t:=\FS\vcenter{\hbox{\SkeinPresA}};
\end{align*}
\begin{align*}\label{eq:skein-presB}
\tag{$J$}\cJ_t:=\FS\vcenter{\hbox{\SkeinPresB}}.
\end{align*}
For the presentations of type (\ref{eq:skein-presA}) and (\ref{eq:skein-presB}) the triangular cell labelled $t(a)$ is \emph{any} non-trivial tree with all vertices coloured by $a$ and the $b$-tree represents a \emph{right-vine} having the same number of leaves as the tree on the left (e.g.,~if $t$ is a single caret, then $\cG_t$ produces the Cleary's golden ratio slope Thompson group).
It was previously proved that these FS categories are Ore and have faithful canonical action.
For each $X\in\{F,T,V\}$ we set $H_n^X,G_t^X,J_t^X$ the FS groups associated to $\cH_n,\cG_t,\cJ_t$, respectively, and write $J^0_t$ for the subgroup of $g\in J_t^F$ acting trivially near the extremities of $\fC$. In \cite{Brothier-Seelig24} we have obtained the somewhat surprising result that the (binary) $F$-type FS group $H_n^F$ is isomorphic to $F_n$: the $F$-type $n$-ary Higman--Thompson group; \cite{Higman74,Brown87}. 
We plan to investigate in the future the relationship between $H_n^T,H_n^V$ and $T_n,V_n$.
Section \ref{sec:isomorphism-problem} is dedicated to proving the following.

\begin{mainthm}[classification of explicit simple groups;~section \ref{sec:isomorphism-problem}]\label{mainthm:B}
Let $\cE$ be the set of derived subgroups of $H_n^F$, $H_n^T$, $H_n^V$, $G_t^F$, $G_t^T$, $G_t^V$, $J_t^0$, $J_t^T$, $J_t^V$ for $t$ any non-trivial uncoloured tree and for $n\geqslant 2$.
Say that the derived subgroups $D(H_n^F)$, $D(G_t^F)$, $D(J_t^0)$ are of type $F$, $D(H_n^T)$, $D(G_t^T)$, $D(J_t^T)$ are of type $T$, and $D(H_n^V),D(G_t^V),D(J_t^V)$ are of type $V$. The following hold.
\begin{enumerate}
\item \textnormal{(Simplicity and finiteness properties).} All the groups in $\cE$ are simple. The groups $D(G_t^T),D(G_t^V),D(J_t^T),D(J_t^V)$ are finitely presented and in fact have the topological finiteness property ${\tt F}_\infty$. On the other hand the groups $D(H_n^F)$, $D(G_t^F)$, $D(J_t^0)$ are not finitely generated. 
\item \textnormal{(Type is remembered).} If $A,B\in\cE$ and $A\simeq B$, then $A$ and $B$ have same type.
\item \textnormal{(Three distinct families for $T$ and $V$-type).} If $X=T$ or $V$, then the three groups $D(H_n^X),D(G_t^X),D(J_s^X)$ are all pairwise non-isomorphic for all choices of $n,t,s.$
\item \textnormal{(Three distinct families for $F$-type).} The three groups $D(H_n^F),D(G_t^F),D(J_s^0)$ are all pairwise non-isomorphic for all choices of $n,t,s.$
\item \textnormal{(Parameter is remembered for $H_n^Y$).} If $Y=F,T,V$, then $D(H_n^Y)\simeq D(H_m^Y)$ implies $n=m.$ 
\item \textnormal{(Parameter is remembered for $G^X_t$ and $J^X_t$).} If $X=T,V$ and $D(G_t^X)\simeq D(G_s^X)$ (resp.~$D(J_t^X)\simeq D(J_s^X)$), then $t,s$ have same number of leaves.
Similar statements hold for the families of groups $D(G_t^F)$ and $D(J_t^0).$
\end{enumerate}
\end{mainthm}
Note that $D(H_n^F)=D(H_n^0)$ and $D(G_t^F)=D(G_t^0)$ while $D(J_t^0)\neq D(J_t^F)$. This explains the statements regarding $F$-type groups.

\subsection*{Acknowledgements}
We warmly thank Matt Brin, Christian de Nicola Larsen, Yash Lodha, Dilshan Wijesena, and Matt Zaremsky for their constructive comments and pertinent questions on a previous version of this article. 

\section{Preliminaries}\label{sec:prelim}

We give a self contained overview of the forest-skein (FS) formalism. 
We invite the interested reader to consult \cite{Brothier22,Brothier23,Brothier-Seelig24} for further details. 
To help the reader we have added an index of notations at the end of this paper.

\subsection{Trees and forests} Let $\{\0,\1\}^*$ be the monoid of all words in $\0$ and $\1$ with concatenation $(u,v)\mapsto u\cdot v$ and empty word $\varnothing$. A \emph{(uncoloured)} \emph{tree} is a finite rooted full binary tree and an $n$-rooted \emph{(uncoloured)} \emph{forest} is an $n$-tuple of trees. The vertices of a tree are parametrised by words in $\{\0,\1\}^*$. The \emph{children} of a vertex $\nu\in\{\0,\1\}^*$ are the vertices $\nu\cdot\0$ and $\nu\cdot\1$. The \emph{leaves} of a tree $t$ are those vertices whose children are not contained in $t$ and all other vertices of $t$ are called \emph{interior}. We write $\Leaf(t)$ and $\Ver(t)$ for these sets. The lexiocographical order on $\{\0,\1\}^*$ (with $\0<\1$) is a total order, so the leaf set of any tree is totally ordered. For a tree $t$ with $n$ leaves we write
\begin{align}\label{eq:address-map}
\ell_t:\{1,2,\dots,n\}\to\Leaf(t), \ i\mapsto\ell_t(i)
\end{align}
for the unique order-preserving bijection, and we often identify these two sets via $\ell_t$. We say $\ell_t(i)$ is the \emph{address} of the $i$th leaf of $t$. The notion of leaves and interior vertices naturally extend to forests by taking disjoint unions, hence we often talk of the number of leaves of a forest, which is the sum of the number of leaves of its trees. For each word $\nu\in\{\0,\1\}^*$ there is a unique smallest tree which contains $\nu$ as a leaf, which is called the \emph{narrow tree containing $\nu$}.
Given a set $S$ of \emph{colours} we may form \emph{coloured trees over $S$} by equipping a tree $t$ with a \emph{colour function} $c:\Ver(t)\to S$, which labels the interior vertices of $t$ with colours from $S$. We may then form \emph{coloured forests over $S$} by taking finite length lists of coloured trees over $S$. 

\begin{figure}[H]
\centering
\vspace{0.5cm}
\begin{subfigure}{.4\linewidth}
    \centering
    \FigA
    \caption{}
    \label{fig:FigA}
\end{subfigure}
\begin{subfigure}{.4\linewidth}
    \centering
    \FigB
    \caption{}
    \label{fig:FigB}
\end{subfigure} 
\begin{subfigure}{.4\linewidth}
    \vspace{0.75cm}
    \centering
    \FigC
    \caption{}
    \label{fig:FigC}
\end{subfigure}
\begin{subfigure}{.4\linewidth}
  \vspace{0.75cm}
  \centering
  \FigD
  \caption{}
  \label{fig:FigD}
\end{subfigure} 
\caption{(\subref{fig:FigA}) An uncoloured tree. (\subref{fig:FigB}) An uncoloured forest. (\subref{fig:FigC}) The narrow tree containing $\0\1\0$. (\subref{fig:FigD}) A coloured forest over $\{a,b,c\}$.}
\label{fig:images}
\end{figure}

The tree with one leaf is called the \emph{trivial tree} and is denoted by $I$. It admits no colour functions as $\Ver(I)$ is empty. The unique tree with two leaves $Y$ is called the \emph{caret}. The caret has a unique interior vertex which may be coloured. We write $Y_a$ for the caret whose interior vertex is coloured $a$. Diagrammatically
\begin{align*}
Y_a=\vcenter{\hbox{\ACaret}}.
\end{align*}
We will also make use of left and right \emph{vines} which are the trees defined inductively by starting with a caret and gluing carets to the first leaf and last leaf, respectively. For example, here are the first three (uncoloured) right vines:
\begin{align}\label{eq:vines}
\rho_1:=\UncolouredCaret,~\rho_2:=\UncolouredVineA,~\textnormal{and}~\rho_3:=\UncolouredVineB.
\end{align}

\subsection{Forest-skein categories} Let $\FS\la S\ra$ be the set of all forests coloured over $S$ and write $f:r\leftarrow l$ if $f\in\FS\la S\ra$ is a forest with $r$ roots and $l$ leaves. 

\subsubsection{Vertical concatentation} If $f:m\leftarrow l$ and $g:l\leftarrow n$, then we may form a new forest $f\circ g:m\leftarrow n$ by gluing the $i$th tree of $g$ to the $i$th leaf of $f$ for all $1\leqslant i\leqslant l$ and reducing.

\begin{figure}[H]
\centering
\Composition\end{figure}

We often write $fg$ instead of $f\circ g$. 
Moreover, if we write $fg$ without qualifying number of roots and leaves of these forests, we are implicitly assuming they are composable.

\subsubsection{Horizontal concatenation} Given \emph{any} two forests $p:k\leftarrow l$ and $q:m\leftarrow n$ we form their \emph{tensor product} $p\ot q:k+m\leftarrow l+n$ which is the concatenation of their lists of trees.

\begin{figure}[H]
\centering
\TensorProd
\end{figure}

By setting the source (resp.~target) of $f:r\leftarrow l$ to be $l$ (resp.~$r$) yields a (small strict) monoidal category $(\FS\la S\ra,\circ,\ot)$ that we call \emph{free FS category over $S$}. 

\subsubsection{Skein relations} A \emph{skein relation over $S$} is a pair of trees $(u,v)$ (often written $u=v$ or $u\sim v$) coloured over $S$ having the same number of leaves. 

\begin{figure}[H]
\centering
\SkeinRelation
\caption{A skein relation over $\{a,b\}$.
}
\label{fig:skein-relation}
\end{figure}

\subsubsection{Forest-skein category} Given a set $R$ of skein relations over $S$ we write $\ov{R}$ for the smallest equivalence relation containing $R$ on $\FS\la S\ra$ closed under $\circ$ and $\ot$.
Hence, if $(x,y)\in \ov R$, then $(f\circ x\circ g,f\circ y\circ g)\in \ov R$ for all $f,g$ composable with $x$ and moreover $(p\ot x \ot q, p\ot y \ot q)\in \ov R$ for all forests $p,q.$
The quotient $$\FS\la S|R\ra:=\FS\la S\ra/\ov{R}$$ inherits the structure of a monoidal category from $\FS\la S\ra$ with the source and target of $[f]$ being the number of leaves and number of roots of $f$, respectively, and
\begin{align*}
[f]\circ[g]:=[f\circ g]\quad\textnormal{and}\quad[p]\ot[q]:=[p\ot q]
\end{align*}
where the number of leaves of $f$ equals the number of roots of $g$ and where $p$ and $q$ are any forests. Since skein relations do not change number of leaves this composition is well-defined.
Here $[f]$ is denoting the congruence class of $f$ in $\FS\la S|R\ra$. 
We usually just write $f$ instead of $[f]$ for a morphism in $\FS\la S|R\ra$ and also refer to it as a \emph{forest} and we may insist that a morphism in the free category $\FS\la S\ra$ is a \emph{free forest}.

\begin{definition}[forest-skein category]
We call $\FS\la S|R\ra$ the \emph{FS category with skein presentation $\la S|R\ra$}.
\end{definition}

We usually denote FS categories by calligraphic letters $\cF,\cG,\cH,\dots$ leaving an underlying skein presentation implicit. For a forest $f:r\leftarrow l$ in an FS category $\cF$ we write 
\begin{align*}
\Root(f):=\{1,\dots,r\}\quad\textnormal{and}\quad\Leaf(f):=\{1,\dots,l\}
\end{align*}
whose elements we call the \emph{roots} and \emph{leaves} of $f$, respectively. We warn the reader that we cannot in general identify the leaves of a tree (i.e.,~a forest $t$ having $|\Root(t)|=1$) in an FS category
with addresses as in (\ref{eq:address-map}) as these words usually change under skein relations (unless the skein relations are \emph{shape-preserving}; see \cite{Brothier23}).

\subsubsection{Larger FS categories} Let $\Sigma_n$ be the $n$th symmetric group. 
We represent elements of $\Sigma_n$ using diagrams in $\R^2$ in a similar manner than the Artin braid groups (see \cite[Sec.~1.6.1]{Brothier22})
We define $\cF^V$ be the set of all pairs $(f,\pi)$, where $f:r\leftarrow l$ is a forest in $\cF$, and $\pi\in\Sigma_l$. We extend the monoidal category structure of $\cF$ to $\cF^V$ as follows. For $f:m\leftarrow l$ and $g:l\leftarrow n$ we set
\begin{align*}
(f,\pi)\circ(g,\tau)=(f\circ g^\pi,\pi^g\circ\tau),
\end{align*}
where the $\pi(i)$th tree of $g^\pi:l\leftarrow n$ is the $i$th tree of $g$, and $\pi^g\in\Sigma_n$ is $\pi$ \emph{cabled} according to the leaves of the trees in $g$. 
This means that if a segment $S$ of the permutation $\pi$ ends at the $i$th root of $g$, then this segment gets replaced by $n$ parallel segments (where $n$ is the number of leaves of the $i$th tree of $g$), like a \emph{cable} of wires, with each segment starting at a different leaf of the $i$th tree of $g$, see the example below:
\begin{figure}[H]
\centering
\Cabelling
\end{figure}
Once again, we often write composition as juxtaposition. Moreover, for all $(p,\xi),(q,\eta)\in\cF^V$ we take
\begin{align*}
(p,\xi)\ot(q,\eta)=(p\ot q,\xi\ot\eta),
\end{align*}
where $\xi\ot\eta$ is horizontal concatenation of permutations (thought of as planar diagrams). 
We identify a forest $f$ with $(f,\id)$ and a permutation $\pi$ with $(I,\pi)$ so that $(f,\pi)=f\circ\pi$ inside $\cF^V$ and we extend the root and leaf sets to $\cF^V$ by $\Root(f\pi)=\Root(f)$ and $\Leaf(f\pi)=\Leaf(f)$.
We now define $\cF^T$ to be the set of all $(f,\pi)\in\cF^V$, where $\pi$ is a cyclic permutation. This is closed under $\circ$, but not $\ot$. We call $\cF^V$ and $\cF^T$ the \emph{$V$-type} and \emph{$T$-type FS categories of $\cF$}, respectively. We have embeddings of categories $\cF\into\cF^T\into\cF^V$.
We may write $\cF^F$ for $\cF.$

\subsection{Poset of trees and monoid of pointed trees} Let $\cF$ be an FS category with \emph{tree set} $\cT:=\{t\in\cF:|\Root(t)|=1\}$. The relation $s\leqslant t$ if and only if there is a forest for which $t=s\circ f$ is a partial order on $\cT$. 
A \emph{pointed tree in $\cF$} is a pair $(t,i)$, where $t$ is a tree in $\cF$ having $n$ leaves, and $1\leqslant i\leqslant n$ is a leaf of $t$ called the \emph{distinguished leaf}. 
Write $\cT_\ast$ for the set of all pointed trees in $\cF$ (called $\cT_p$ in \cite{Brothier-Seelig24}), which is a monoid with the following structure: let $(t,i)\circ(s,j)\in\cT_*$ be formed by gluing the root of $s$ to the $i$th leaf of $t$ and keeping the distinguished leaf of $s$. The identity element is the unique pointed trivial tree. The set $\cT_o$ (resp.~$\cT_\omega$) of pointed trees in $\cT_\ast$ with first (resp.~last) distinguished leaf is a submonoid. 
When the equivalence class of a tree in $\cF$ consists of trees with the same shape, then it makes sense to identify $(t,i)$ with $(t,\ell_t(i))$, where $\ell_t(i)\in\{\0,\1\}^*$ is the address of $i$ in $t$. This is the case for any any tree in a free FS category or any caret in any FS category.

\subsection{Forest-skein groups} 

\subsubsection{Ore categories} Let $\cD$ be a (small) category. We say $\cD$
\begin{itemize}
    \item is \emph{left-cancellative} if whenever there is an equality morphisms $f\circ p=f\circ q$, we have $p=q$; and
    \item has \emph{common right-multiples} if for any morphisms $g,h$ with the same target, there exist morphisms $k,l$ such that $g\circ k=h\circ l$.
\end{itemize}
If $\cD$ satisfies both of these we call $\cD$ an \emph{Ore} category. 
An FS category $\cF$ has either of these properties if and only if $\cF^T$ or $\cF^V$ has it.
Moreover, for FS categories it is sufficient to check the above axioms for $f$, $g$, $h$ being \emph{trees}. 
In this way, an FS category $\cF$ has common right-multiples if and only if its poset of trees $(\cT,\leqslant)$ is \emph{directed}.

\subsubsection{Fraction groupoids and groups} Let $\cD$ be an Ore category and write $\cP_\cD$ for the set of all pairs $(f,g)$ of morphisms having the same target. Consider the equivalence relation $\sim$ on $\cP_\cD$ generated by all $(f,g)\sim(f\circ p,g\circ p)$ and write $[f,g]$ for the class of $(f,g)$. 
The category $\cD$ being Ore implies the quotient $\Frac(\cD):=\cP_\cD/\sim$ is a groupoid with
\begin{align*}
[f,g]\cdot[g,h]=[f,h]\quad\textnormal{and}\quad[f,g]^{-1}=[g,f].
\end{align*}
When $\cD$ is (both-sided) cancellative $f\mapsto [f,\id]$ defines an injection $\cD\into\Frac(\cD)$ and $[f,g]$ can be identified with $f\circ g^{-1}.$
We call this the \emph{fraction groupoid of $\cD$}, and due to this constuction we say that Ore categories $\cD$ admit a \emph{calculus of (right-)fractions}. The assignment $\cD\mapsto\Frac(\cD)$ is functorial and extends naturally to the \emph{universal enveloping groupoid functor} $U$, which is defined on all (small) categories and acts by formally inverting every morphism in a given category. For every object $d$ in an Ore category $\cD$ the \emph{isotropy group} $\Frac(\cD,d)$ is the subgroup of $\Frac(\cD)$ consisting of classes $[f,g]$ where the source of $f$ and $g$ are $d$.

\begin{definition}[forest-skein groups]
Let $\cF$ be an Ore FS category and $X=F,T,V$. The isotropy group $\Frac(\cF^X,1)$ is called the \emph{$X$-type FS group of $\cF$}. 
\end{definition}
Hence, elements of $\Frac(\cF,1)$ are represented by pairs of \emph{trees} $(t,s)$ in $\cF$ having the same number of leaves. We think of the fraction $[t,s]$ diagrammatically: take the tree diagram $t$ and stack the tree diagram for $s$ upside down on top of $t$ so that its leaves are glued in order to those of $t$. Here are some examples of tree pair diagrams:
\begin{align*}
\generators
\end{align*}
For $X=T,V$, elements of $\Frac(\cF^X,1)$ can always be represented by $(t\pi,s)$, where $t,s$ are trees and $\pi$ is a permutation of their leaves. In this article we identify such pairs $(t\pi,s)$ with triples $(t,\pi,s)$ for no other reason than aesthetics. 
This represents the element $[t,\pi,s]=t\circ\pi\circ s^{-1}$ when taking formal inverses which can be drawn similar to above now gluing leaves of $s$ to the leaves of $t$ according to the permutation $\pi$.

\begin{notation}
We write $G^F$, $G^T$, $G^V$ for the FS groups of an Ore FS category $\cF$ suppressing the dependence on $\cF$ when the context is clear. We usually drop ``$F$'' superscripts and write $G=G^F$ to make notations lighter.     
\end{notation}

Deciding if a given category is Ore is a difficult question. Fortunately, we have the following key result; see \cite[Theo.~8.3]{Brothier22}.

\begin{theorem}[large family of explicit Ore FS categories]\label{theo:Ore-categories}
Let $x$ and $y$ be \emph{any} uncoloured trees with the same number of leaves. Then $\cF=\FS\la a,b|x(a)=y(b)\ra$ is an \emph{Ore} FS category where $x(a)$ denotes the tree where all interior vertices are coloured by $a.$
\end{theorem}

\subsection{Properties of forest-skein groups} 
For this section fix an Ore FS category $\cF=\FS\la S|R\ra$ with FS groups $G$, $G^T$, $G^V$, and pointed tree monoid $\cT_\ast$.

\subsubsection{Canonical actions for pointed trees and FS groups} From $\cF$ we construct a totally ordered profinite space $\fC$ as well as actions $\al:G^V\act\fC$ and $\beta:\cT_\ast\act\fC$, which we call the \emph{canonical actions of $\cF$}. We will recall these constructions in detail in section \ref{sec:rigid-FS}, but until then the reader can think of $\al:G^V\act\fC$ like the defining action of Thompson's group $V\act\{\0,\1\}^\N$ on the Cantor space of infinite binary sequences, and $\beta:\cT_\ast\act\fC$ as the action $\{\0,\1\}^*\act\{\0,\1\}^\N$ by adding a finite word to a the start of a sequence. However, in general the action $\al:G^V\act\fC$ is \emph{not always faithful}. 

\subsubsection{Embeddings of Thompson's groups} 

For each colour $a\in S$ we have an injective morphism $C_a:\FS\la*\ra\into\cF$ given by colouring each interior vertex by $a$ (see corollary 3.9 of \cite{Brothier22}), which induces an embedding $C_a:X\into G^X$ (when $X$ is either of the Thompson groups $F,T,V$) which is equivariant with respect to the canonical actions; see \cite[Prop.~4.1]{Brothier-Seelig24}). This is called the \emph{$a$-vertex embedding}.

\subsubsection{Simplicity}\label{sec:simplicity} The collection of all FS categories forms a category where the morphisms are monoidal functors preserving number of leaves. A \emph{quotient} of an FS category is then a surjective morphism. We say the Ore FS category $\cF$ is \emph{simple} if any proper quotient of $\cF$ is not left-cancellative. 
Since having common right-multiple is preserved under quotient we have that $\cF$ is simple if and only if it does not admit any proper Ore quotient. Note this is \emph{not} the usual definition of a simple object in a category. 

For the canonical group action $\al:G^V\act\fC$ we write $o:=\min(\fC)$ and $\omega:=\max(\fC)$ for the extremities of $\fC$ (as a totally ordered space) and write $G^0\lhd G^F$ for the normal subgroup acting trivially near $o$ and $\omega.$ Note that $D(F)=F^0$ but in general $D(G^F)\neq G^0.$ 
Though, we have the following result which follows from \cite[Theo.~3.5]{Brothier-Seelig24}.

\begin{proposition}\label{prop:D(G)}
If the germ groups of $G$ at $o$ and $\omega$ are abelian, then $D(G)=D(G^0).$ The converse also holds.
\end{proposition}

Here is the main conceptual result of \cite{Brothier-Seelig24} connecting two notions of simplicity with dynamics.

\begin{theorem}[simplicity theorem]\cite[Theo.~3.5, 3.7, 3.10, and 4.3]{Brothier-Seelig24} \label{theo:simplicity}
Let $\cF$ be an Ore FS category with FS groups $G$, $G^T$, $G^V$, and let $G^0\lhd G$ be as above. Consider the following assertions.
\begin{enumerate}
 \item The Ore FS category $\cF$ is \emph{simple}.
    \item The derived subgroup $D(G^0)$ is \emph{simple}.
    \item The derived subgroup $D(G^T)$ is \emph{simple}.
    \item The derived subgroup $D(G^V)$ is \emph{simple}.
    \item The canonical group action $G\act\fC$ is \emph{faithful}.
    \item The canonical group action $G^V\act\fC$ is \emph{faithful}.
\end{enumerate}
The assertions (2,3,4,5,6) are equivalent. 
The assertion (1) implies (2). If $\cF$ is moreover cancellative (on both side), then (2) implies (1).
\end{theorem}

\subsubsection{Topological finiteness properties}\label{sec:finiteness-properties} Recall a group $\Gamma$ \emph{is of type $\mathtt{F}_\infty$} if it admits an Eilenberg-MacLane $K(\Gamma,1)$ complex with finite $n$-skeleton for all $n$. All FS groups can be interpreted as \emph{operad groups} of Thumann \cite{Thumann17}, and under a mild hypothesis on the FS category, we may apply Thumann's topological finiteness result \cite[Theo.~4.3]{Thumann17} to deduce many FS groups are of type $\mathtt{F}_\infty$. The following can be found in \cite[Cor.~7.5]{Brothier22}.

\begin{theorem}[topological finiteness properties for FS groups]\label{theo:finiteness-properties}
For any uncoloured trees $x,y$ with the same number of leaves, the FS groups of $\FS\la a,b|x(a)=y(b)\ra$ are of type $\mathtt{F}_\infty$.
\end{theorem}

\subsubsection{Abelianisation}\label{sec:abelianisation} We greatly care about derived subgroups of FS groups as they may be simple; see theorem \ref{theo:simplicity}. To deduce finiteness properties for these groups given the above results, it suffices to witness finiteness of the abelianisation. Here is a useful result for the $T$ and $V$-cases. Let $\Z[X]$ be the set of formal sums of elements of a set $X$ with coefficients in $\Z$, which is free abelian with basis $X$. For any free forest $f$ over a colour set $S$ and any fixed colour $a\in S$ let $\chi_a(f)\in\Z[S\setminus\{a\}]$ be the sum of all colours appearing in $f$ except for any instances of $a$. For example
\begin{align*}
\chi_a(\vcenter{\hbox{\FigD}})=b+c.
\end{align*}

This map provides the following practical way to compute abelianisations of $T$ and $V$-type FS groups given a skein presentation for them. In essence the abelianisation of these groups is given by ``counting colours except for $a$ modulo skein relations''.

\begin{theorem}[abelianisation of $T$ and $V$-type FS groups]\cite[Theo.~5.1]{Brothier-Seelig24} \label{theo:abelianisation}
Let $\cF=\FS\la S|R\ra$ be a presented Ore FS category and let $X=T,V$. Fix a colour $a\in S$. 
The following group morphism
\begin{align*}
G^X\onto\frac{\Z[S\setminus\{a\}]}{\la\chi_a(u)-\chi_a(v):(u,v)\in R\ra} \ , \ [t,\pi,s]\mapsto\chi_a(t)-\chi_a(s)
\end{align*}
induces an isomorphism on the abelianisation $G^X_{ab}$.
\end{theorem}

Recall that $G^0$ is the subgroup consisting of those $g\in G$ that act trivially near the two extremities $o,\omega$ of the canonical space $\fC.$
A similar argument to the one of \cite[Theo.~5.1]{Brothier-Seelig24} implies the following.

\begin{corollary}\label{cor:abelianisation-Gzero}
If the canonical action is faithful, then the abelianisation of $G^0$ is isomorphic to the abelianisation of $G^T.$
\end{corollary}

\begin{proof}[Sketch of proof.]
For the convenience of the reader we sketch a proof of this result.
By definition we have an obvious embedding $G^0\into G^T$ and thus obtain a morphism $G^0\to G^T_{ab}$ by composing with $G^T\onto G^T_{ab}.$
If we take any colour $b\in S$ and any tree $t$ with three leaves we may form $t\circ (I\ot Y_b\ot I)\circ (I\ot Y_a\ot I)^{-1}\circ t^{-1}$ which is in $G^0.$
This element is sent to $b$ inside $G^T_{ab}.$ Since the $b$'s generate $G^T_{ab}$ by theorem \ref{theo:abelianisation} this shows that the composition $G^0\to G^T_{ab}$ is surjective. By universality we deduce a surjective morphism $\varphi:G^0_{ab}\onto G^T_{ab}.$ 
Let us explain why $\varphi$ is injective.
Consider $g\in G^0.$ 
Since $g$ acts trivially on the extremities there exists a large enough tree $t$ with $n$ leaves so that $g$ decomposes into $g=t\circ (I\ot B\ot I)\circ t^{-1}$ where $B$ is an element of the fraction groupoid $\cF$ with $n-2$ roots and $n-2$ leaves. 
Here $B$ stands for "block" since it has same number of roots and leaves.
The block $B$ is a composition of elements of the fraction groupoid with at most one caret. 
This means that $I\ot B\ot I$ is a finite product of $b_{j,n}^{\pm1}$ with $2\leqslant j\leqslant n-1$ and $b\in S$ (where $b_{j,n}$ stands for the forest with $n$ roots and a single caret of colour $b$ at the $j$th root).
We obtain that $g$ is a finite product of element of the form $t\circ b_{j,n}\circ a_{i,n}^{-1}\circ t^{-1}$ or their inverse with $b\in S$ and $2\leqslant i,j\leqslant n-1.$
Using this fact and existence of common-right multiples we deduce that if $\chi_a(h)=b$, then $h$ is conjugated by $D(G^0)$ to $t\circ b_{2,n}\circ a_{2,n}^{-1}\circ t^{-1}$. Then, if $\chi_a(g)=0$ one can use this fact to arrange for $g$ to be trivial up to multiplying by $D(G^0)$. 
This implies that $\varphi$ is injective.
\end{proof}

\begin{example}For instance, if $S=\{a,b\}$ and $R$ consists of the single skein relation 
\begin{align*}
\SkeinRelation,
\end{align*}
then the abelianisation of the associated FS groups $G^0,G^T,G^V$ are all isomorphic. They have one generator $b$ subject to the relation $3b=0$. Hence, $G^0_{ab}\simeq G^T_{ab}\simeq G^V_{ab}\simeq \Z/3\Z.$
\end{example}

We often have that $D(G^T)$ and $D(G^V)$ are finitely presented.
However, $D(G^0)$ is never finitely generated as observed below.

\begin{proposition}\label{prop:not-fg}
The group $D(G^0)$ is not finitely generated.
\end{proposition}

\begin{proof}
Consider an Ore FS category $\cF$ with FS group $G$ and canonical action $G\act\fC.$ Recall that $o,\omega$ are the two extremities of $\fC$ and that $G^0$ is the normal subgroup of $G$ acting trivially near $o$ and $\omega.$
Let $X\subset D(G^0)$ be a finite subset and write $H$ for the subgroup generated by $X$.
By definition of $G^0$ for each $g\in X$ there is an open subset $\Omega_g\subset \fC$ containing $o$ and $\omega$ on which $g$ acts like the identity. 
Set $\Omega:=\cup_{g\in X} \Omega_g$. It is clear that $H$ is contained in the subgroup of $D(G^0)$ whose element act like the identity on $\Omega$.
By taking uncoloured trees and colouring all there nodes by a fixed colour of $\cF$ we obtain an embedding $F\into G^F$ sending $D(F)$ inside $G^0$ and thus $D(D(F))$ inside $D(G^0).$
Since $D(F)$ is simple and non-abelian $D(D(F))=D(F)$. 
Additionally, under this identification on can show that elements of $D(F)$ can act non-trivially arbitrarily close to $o\in\fC$.
This implies that $D(F)\not\subset H$ and thus $D(G^0)\not\subset H.$
Therefore, $D(G^0)$ is not finitely generated.
\end{proof}

\subsection{Rigid actions and reconstruction theorems}

\subsubsection{Some terminology for ordered spaces}

Fix a set $Z$. 
A \emph{total order} or \emph{linear order} $\leqslant$ on $Z$ is a binary relation that is transitive, anti-symmetric, so that two different elements are always comparable. We write $<$ for \emph{strict} orders (when we exclude equality) and $\leqslant$ for non-strict ones. 
The order is \emph{dense} when given $x<y$ there exists $z$ satisfying $x<z<y.$
A \emph{jump} is a pair $(x,y)$ with $x<y$ and satisfying that $x\leqslant z\leqslant y$ implies $z\in\{x,y\}.$
Hence, not having jumps means being dense.
Endpoint stands for minimum and maximum and intervals (open, closed) are defined in the obvious way. We will always assume that a totally ordered space $(Z,\leqslant)$ is equipped with its \emph{order topology}: the topology generated by open intervals.
We define Dedekind completions for spaces without endpoints. 
\begin{definition}
Let $(Z,\leqslant)$ be a dense totally ordered space without endpoints.
\begin{enumerate}
\item A \emph{Dedekind cut} of $(Z,\leqslant)$ is a non-empty proper subset $A\subset Z$ so that if $x\leqslant a$ with $a\in A$, then $x\in A$ and if $a\in A$ then there exists $b\in A$ so that $a<b$.
\item \emph{The Dedekind completion} $(\Ded(Z),\leqslant)$ of $(Z,\leqslant)$ is the set of Dedekind cuts ordered by inclusion. This is again dense, totally ordered, without endpoints.
\end{enumerate}
\end{definition}

We briefly recall some group action terminology.

\begin{definition}
Consider a group action $\Ga\act Z$ on a set $Z$.
\begin{itemize}
\item The \emph{support} of $g\in \Ga$ is the set $\supp(g):=\{x\in Z:\ gx\neq x\}.$
\item If $U\subset Z$ is a subset, then $\Ga_U:=\{g\in \Ga:\ \supp(g)\subset U\}$.
\end{itemize}
Assume now that $\leqslant$ is a total order that $\Ga$ preserves.
\begin{itemize}
\item An element $g\in\Ga$ has \emph{bounded support} or is \emph{bounded} if there exists $x,y\in Z$ so that $\supp(g)\subset (x , y).$
\item The action $\Ga\act (Z,\leqslant)$ is \emph{$n$-o-transitive} (for $n\geqslant 1$) if given any strict chains $x_1<\cdots<x_n$ and $y_1<\cdots<y_n$ inside $X$ there exists $g\in\Ga$ satisfying $g(x_i)=y_i$ for all $1\leqslant i\leqslant n.$
It is \emph{highly-o-transitive} when it is $n$-transitive for all $n\geqslant 1.$
\end{itemize}
Assume now that $\tau$ is a topology on $Z$ and that $\Ga$ act by homeomorphisms.
\begin{itemize}
\item The action is \emph{locally dense} when given any open subset $U\subset Z$ and point $p\in U$ we have that the closure of the orbit $\Ga_U\cdot p$ is a neighbourhood of $p$.
\end{itemize}
\end{definition}

\subsubsection{Some terminology for cyclically ordered spaces}

A \emph{cyclic order} $Cr$ on a set $Z$ is a subset of $Z^3$ satisfying the axioms
\begin{enumerate}
\item $(x,y,z)\in Cr \text{ implies } (y,z,x)\in Cr$;
\item $(x,y,z)\in Cr \text{ implies } (x,z,y)\notin Cr$;
\item if $(x,y,z)$ and $(z,t,x)$ are in $Cr$, then $(y,z,t)$ is in $Cr$ and
\item given $x,y,z\in Z$ we either have $(x,y,z)\in Cr$ or $(x,z,y)\in Cr.$
\end{enumerate}
The last axiom is not always assumed by other authors. It corresponds in having $Cr$ \emph{total}.
Additionally, the second axiom implies \emph{strictness}. If $<$ is a strict total order on $Z$, then 
$$Cr:=\{(x,y,z)\in Z^3:\ x<y<z \text{ or } y<z<x \text{ or } z<x<y\}$$
is a cyclic order.
Conversely, if $Cr$ is a cyclic order and $x_0\in Z$ is a fixed point, then we can \emph{unroll} $Z$ at $x_0$ obtaining a strict totally ordered space $(Z,<)$ so that $x_0$ is the smallest element and where given $y\neq x_0\neq z$ we have $y<z$ if and only if $(x_0,y,z)\in Cr$. 
(For example, we can unroll a circle into a half open real interval). 
We define in the analogous way than for total orders the notions of intervals, density, order-topology, and Dedekind completion for cyclic orders. 
We will say that a group action $\Ga\act (Z,Cr)$ is \emph{order-preserving} if $(x,y,z)\in Cr$ implies $(gx,gy,gz)\in Cr$ for all $g\in\Ga.$

\subsubsection{Rigidity theorems for group actions}
We now enunciate three rigidity theorems. Let $\gamma:\Ga\act (Z,\tau)$ be a group $\Ga$ acting by homeomorphism on a topological space $(Z,\tau)$.
The action is called \emph{T-rigid} (T standing for topology) when 
\begin{itemize}
\item $(Z,\tau)$ is locally compact, Hausdorff, and has no isolated points;
\item $\gamma$ is faithful and is \emph{locally dense}. 
\end{itemize}
When $\Ga$ admits a T-rigid action we say that $\Ga$ is a \emph{T-rigid group}.

The following theorem is due to Rubin \cite[Cor.~3.5]{Rubin89}.
We also recommend the paper \cite{Belk-Elliott-Matucci23} for a short and elementary proof.

\begin{theorem}\label{theo:T-rigid}
If $\gamma_i:\Ga_i\act (Z_i,\tau_i), i=1,2$ are T-rigid actions, then for any isomorphism $\theta:\Ga_1\to \Ga_2$ there exists a unique homeomorphism $f:Z_1\to Z_2$ satisfying that 
$$f\circ \gamma_1(g)=\gamma_2(\theta(g))\circ f \text{ for all } g\in \Ga_1.$$
Hence, a T-rigid group admits a unique T-rigid action up to equivariant homeomorphism. 
\end{theorem}

Similarly, there are the classes of O-rigid (O for order) and C-rigid (C for cyclic order) groups and actions. 
If $\Ga$ is O-rigid (resp.~C-rigid), then there exists an O-rigid (resp.~C-rigid) action $\gamma:\Ga\act (Z,\leqslant)$ (resp.~$\Ga\act (Y,Cr)$) on a \emph{Dedekind complete} space that is order-preserving. If $\theta:\Ga_1\to \Ga_2$ is a group isomorphism between two O-rigid groups (resp.~C-rigid groups), then there exists a unique bijection $f:Z_1\to Z_2$ that is either \emph{order-preserving} or \emph{order-reversing} satisfying:
$$f\circ \gamma_1(g)=\gamma_2(\theta(g))
\circ f \text{ for all } g\in \Ga_1.$$
In particular, if $\Ga$ is O-rigid (resp.~C-rigid), then one can reconstruct $(Z,\leqslant)$ (resp.~$(Y,Cr)$), up to reversing the order, and the action of $\Ga$ on it. We now enunciate two rigidity theorems with non-optimal hypothesis that are sufficient for our needs. The following theorem is due to McCleary \cite[Theo.~4]{McCleary78}; see also \cite[Theo.~1]{McCleary-Rubin05}.
\begin{theorem}\label{theo:O-rigid}
Let $(Z,\leqslant)$ be a dense totally ordered set without extremities.
Let $\Ga\act (Z,\leqslant)$ be a faithful, order-preserving action that is $2$-o-transitive and has one non-trivial bounded element.
Then the action of $\Ga$ on the Dedekind completion of $(Z,\leqslant)$ is \emph{O-rigid}.
\end{theorem}

Here is a cyclic version of this theorem due to McCleary and Rubin \cite[Theo.~3]{McCleary-Rubin05}.
\begin{theorem}\label{theo:C-rigid}
Let $(Z,Cr)$ be a dense cyclically ordered space. 
Let $\Ga\act (Z,Cr)$ be a faithful order-preserving action that is $3$-o-transitive.
Then the action of $\Ga$ on the Dedekind completion of $(Z,Cr)$ is \emph{C-rigid}.
\end{theorem}

\begin{remark}
If a group $\Ga$ is O-rigid or C-rigid then it remembers a dynamic on a space with an order (or cyclic order) up to reversing. Taking the order topology (which is invariant under reversing) we obtain a topological space. 
Hence, an O-rigid or a C-rigid group remembers more than a T-rigid group. 
All FS groups coming from FS categories with a faithful canonical action are T-rigid. We will prove more by showing that in the $F$-case (resp.~$T$-case) the group is O-rigid (resp.~C-rigid).
\end{remark}

\section{Rigid actions for forest-skein groups}\label{sec:rigid-FS}

Fix an Ore FS category $\cF$ with tree set $\cT$, pointed tree monoid $\cT_\ast$, and FS groups $G^F$, $G^T$, $G^V$.
We recall the construction of the canonical group and monoid actions of $\cF$ (and their completions). 
Further details can be found in \cite{Brothier22,Brothier-Seelig24}. We will then turn the canonical \emph{group} action into three rigid actions. 
Note that the construction of the canonical group action is functorial (from FS categories to group actions) while the construction of rigid action (from groups to group actions) is \emph{not}, for example the inclusion $T\into V$ does not provide an equivariant continuous map from the circle to the Cantor space $\R/\Z\to\{\0,\1\}^\N$ as such a map is necessarily constant.

\subsection{Canonical actions}\label{sec:canonical-actions}

\subsubsection{Dyadic rational analogue} 

Given two trees $t,s\in\cT$ satisfying $s\leqslant t$ there exists a unique forest $f\in\cF$ so that $t=s\circ f.$
Define the map $\Phi(f):\Leaf(s)\to\Leaf(t)$ so that $\Phi(f)(\ell)$ is the first leaf of the $\ell$th tree of $f$.
This defines a directed system having for limit $$\fQ:=\varinjlim_{t\in\cT}\Leaf(t)=\frac{\cup_{t\in\cT}\{t\}\times\Leaf(t)}{\sim}$$ 
where $\sim$ is generated by $(t,j)\sim(tf,\Phi(f)(j))$ for all $t,f\in\cF$.
We obtain a group action $\al:G^V\act\fQ$ that satisfies the formula:
$$[t,\pi,s] \cdot [s,j]:= [t,\pi(j)],$$
where $t,s$ are trees, $\pi$ is a permutation of their leaves, and $[s,j]$ denotes the class of the $j$th leaf of the tree $s$ inside $\fQ.$ We may refer to $\al:G^V\act\fQ$ as the \emph{canonical group action of $\cF$.}
We will often use the notation $\ell^f:=\Phi(f)(\ell)$ to denote the numbering of the first leaf of the $\ell$th tree of a forest $f$. Equivalently, $\fQ$ can be viewed as the quotient of the pointed trees $\cT_\ast=\cup_{t\in\cT}\{t\}\times\Leaf(t)$ by the equivalence $\sim$ generated by all $(t,\ell)\sim(tf,\ell^f)$. From this perspective, we obtain a natural action $\beta:\cT_\ast\act\fQ$ of the pointed tree monoid $(\cT_*,\circ)$ which factors through the self action $\cT_\ast\act\cT_\ast$ by left-multiplication:
\begin{align*}
(t,i)\cdot[s,j]:=[(t,i)\circ(s,j)].
\end{align*}
We may refer to $\beta:\cT_*\act\fQ$ as the \emph{canonical monoid action of $\cF$.}

\subsubsection{Total order}

Given $[t,i],[s,j]\in\fQ$ we write $[t,i]\leqslant[s,j]$ when there exist forests $f,g\in\cF$ satisfying $t\circ f=s\circ g$ and $i^f\leqslant j^g.$
This defines a dense total order on $\fQ$ having one extremity $o$ equal to its minimum. Note that $o=[t,1]$ for all trees $t$ (the class of a tree pointed at its first leaf).
Moreover, $G\act (\fQ,\leqslant)$ is order-preserving and $G^T\act (\fQ,Cr)$ preserves the cyclic order $Cr$ obtained from $\leqslant.$ Similarly, $\cT_*\act(\fQ,\leqslant)$ is order-preserving.
We write $\fQ^*$ for $\fQ\setminus\{o\}$: $\fQ$ minus its minimum.

\subsubsection{Cantor space analogue} 
By reversing the arrows we obtain a completion of $\fQ$. Given two trees $s\leqslant t$ with $t=s\circ f$ define 
$\Psi(f):\Leaf(t)\to\Leaf(s)$ so that for any leaf $\ell$ of $t$ we have $\Psi(f)(\ell)$ is the unique leaf of $s$ which is a descendant of $\ell$.
This forms an inverse system whose limit 
$$\fC:=\varprojlim_{t\in\cT} \Leaf(t)= \{x\in\prod_{t\in\cT}\Leaf(t) : \ \Psi(f)(x(tf))=x(t) \text{ for all } t\in\cT,\  f\in\cF\}$$ 
is a profinite space that we may call the \emph{canonical space of $\cF$}.
There is a unique action by homeomorphisms $\al:G^V\act \fC$ satisfying 
$$([t,\pi,s]\cdot x)(t)=\pi(x(s))$$
with $x=(x(t))_{t\in\cT}$. 
We write $x\leqslant y$ when $x(t)\leqslant y(t)$ eventually in $t\in\cT.$
This defines a total order on $\fC$ and the order topology coincides with the profinite topology.
We write 
$$o:=\min\fC \text{ and } \omega:=\max\fC$$
for its extremities. We will explain how the pointed tree monoid acts on $\fC$ in a moment. 

\subsubsection{Embeddings of $\fQ$ into $\fC$}

Recall $\fQ^*=\fQ\setminus\{o\}$. There is a unique embedding 
$$\iota^+:\fQ^*\into \fC, \ q\mapsto q^+$$ 
satisfying that $\iota^+([t,j])=x$
where $x(t\circ h)$ is the first leaf of the $j$th tree of $h$ for any forest $h$. We define a second embedding 
$$\iota^-:\fQ^*\into \fC, \ q\mapsto q^-$$
defined as follows. Given $q\in \fQ^*$ the interval $[q^+,\omega]$ is clopen in $\fC$ and has clopen complement $[o,q^+)$.
We set $q^-:=\sup([o,q^+))$ the supremum of $[o,q^+)$. This indeed exists since $[o,q^+)$ is compact.
We have the following descriptions in terms of orbits:
$$G^V\cdot o = \iota^+(\fQ^*)\cup\{o\}=:\fQ^+ \text{ and } G^V\cdot \omega = \iota^-(\fQ^*)\cup\{\omega\}=:\fQ^-.$$
Hence, $\fQ^+$ and $\fQ^-$ are analogous to all those binary sequences which are eventually all $\0$'s and all $\1$'s, respectively. 
By construction, $(q^-,q^+)$ is a jump for each $q\in \fQ^*$ and the maps $\iota^-,\iota^+$ are order-preserving with dense range.
Moreover, $(\omega,o)$ is a jump for the circular order $Cr$ obtained from $\leqslant.$
For convenience, we may extend $\iota^+$ to the extremity of $\fQ$ so that $\iota^+([t,1])=o$ for each tree $t$ pointed at its first leaf.
Additionally, if we identify $\fQ$ with a subset of $\fC$ it will be always via $\iota^+.$

One can characterise $G$ and $G^T$ inside $G^V$ using these orders. The following can be found in \cite[Prop.~6.1]{Brothier22}.

\begin{proposition}[characterising FS type with orders]\label{prop:FS-group-order}
Let $\leqslant$ be the total order of $\fC$ and $Cr$ the cyclic order deduced from it.
The subgroup of $g\in G^V$ that preserves the order $\leqslant$ (resp.~the cyclic order $Cr$) is equal to $G$ (resp.~$G^T$).
\end{proposition}

\subsubsection{Canonical monoid action, cones, and compatibility of monoid and group actions}\label{sec:topology-and-compatibility}

{\bf Canonical monoid action.}
Recall that $\cT_*$ is the monoid of pointed trees from $\cF$ with the gluing operation which acts naturally on $\fQ$ via $\beta:\cT_*\act\fQ$. 
Identifying $\fQ$ as a dense subspace of $\fC$ via $\iota^+$, we obtain that the action $\beta(t,i):\fQ\to\fQ$ is continuous for all $(t,i)\in\cT_*$, so can be uniquely extended to $\beta(t,i):\fC\to\fC$. From this we obtain a monoid action $\beta:\cT_*\act\fC$.

{\bf Cones.}
Given a tree $t$ and a leaf $j$ we consider the \emph{cone}
$\Cone(t,j):=\beta(t,j)(\fC)$
inside $\fC$ equal to the closure of 
$\beta(t,j)(\fQ)=\{[(t,j)\circ (z,i)]: (z,i)\in\cT_\ast\}.$
We note that $\Cone(t,j)$ is equal to the half-open interval
$$\{z\in\fC: [t,j]^+\leqslant z < [t,j+1]^+\}$$
and also equal to the closed interval 
$$\{z\in\fC: [t,j]^+\leqslant z \leqslant [t,j+1]^-\}$$
as $[t,j]^+$ and $[t,j+1]^+$ are the images of $o$ and $\omega$ under the order-preserving map $\beta(t,j)$, respectively. 
The cones form a clopen basis of the topology of $\fC.$ We often write elements of $\Cone(t,j)$ as $(t,j)\cdot w:=\beta(t,j)(w)$ hence interpreting them as having for ``prefix" $(t,j)$. 
Any tree $t$ defines a partition $(\Cone(t,\ell):\ell\in\Leaf(t))$ of $\fC$ into finitely many clopen sets.

{\bf Connection between actions.} 
The canonical group action $\al:G^V\act\fC$ is \emph{locally determined} by the canonical monoid action $\beta:\cT_\ast\act\fC$. Indeed
\begin{align*}
\al([t,\pi,s])\restriction_{\Cone(s,j)}=\beta(t,\pi(j))\circ\beta(s,j)^{-1}, \ (s,j)\cdot w\mapsto (t,\pi(j))\cdot w.
\end{align*}
for all $[t,\pi,s]\in G^V$ and $j\in\Leaf(s)$. 

\begin{definition}[FS-affine]\label{def:FS-affine}
A partial transformation $\phi:\fC\to\fC$ is called \emph{FS-affine} if it has for domain and codomain some cones $\Cone(s,j),\Cone(t,i)$ and is of the form
$$(s,j)\cdot w\mapsto (t,i)\cdot w.$$
\end{definition}

\subsubsection{Pruned relations}\label{sec:pruned}
Pruned relations are relations between free pointed trees that are important for understanding the canonical actions. 
We can \emph{prune} any free tree $t$ at its $i$th leaf. This is the pointed tree $\prune(t,i):=(t',i')$ obtained by removing all carets that are not on the path from the root of $t$ to its $i$th leaf. We use the term ``prune'' as we are removing branches to the tree $t$.
Consider two free trees $u,v$ that are equivalent (i.e.,~they define a skein relation).
If $u$ has $n$ leaves, then for each $1\leqslant i\leqslant n$ we say that $\prune(u,i)$ is prune equivalent to $\prune(v,i).$
Let $\sim_{prune}$ be the equivalence relation on free pointed trees generated by it.
A useful result (see \cite[Prop.~2.3]{Brothier-Seelig24}) is that if $(t,i)$ and $(s,j)$ are free pointed trees that are prune equivalent, then they act in the same way on $\fC$, i.e.,~$\beta(t,i)=\beta(s,j)$.

\subsection{Rigid actions for FS groups}

Write $(\fR^V,\tau)$ for the space $\fC$ equipped with its topology $\tau$.
We will show that $G^V\act (\fR^V,\tau)$ is rigid.
However, the restriction of this action to $G$ or even $G^T$ is not. 
We will easily construct two spaces from $(\fR^V,\tau)$ on which $G$ and $G^T$ act in a rigid way.
We start by a lemma.

\begin{lemma}
The jumps of $(\fC,\leqslant)$ are exactly the $(q^-,q^+)$ with $q\in \fQ^*$.
\end{lemma}

\begin{proof}
Consider a jump $(x,y)$ in $(\fC,\leqslant).$
By definition of the order-topology we deduce that $[o,x]$ is clopen with complement $[y,\omega]$.
By the definition of the profinite topology $[y,\omega]$ can be partitioned in a finite union of cones.
This implies that $y=q^+$ for a certain $q\in \fQ^*.$
Necessarily, since $(x,q^+)$ is a jump we must have $x=q^-.$
\end{proof}

Write $\fC^*$ for $\fC$ without its two extremities $o,\omega$ and recall that $\fQ^*:=\fQ\setminus\{o\}$.
Define the equivalence relation $\sim$ on $\fC$ so that $q^-\sim q^+$ for all $q\in \fQ^*$ and moreover $o\sim \omega.$
Denote by $\fR=\fR^F$ and $\fR^T$ the quotient spaces $\fC^*/\sim$ and $\fC/\sim$, respectively.

\begin{lemma}
The space $(\fR,\leqslant)$ is isomorphic to the Dedekind completion of $(\fQ^*,\leqslant)$.
Moreover, there exists an isomorphism $f:\fR\to \Ded(\fQ^*)$ that restrict to the identity on $\fQ^*$ with respect to the obvious inclusions. 
Similarly, $(\fR^T,Cr)$ is isomorphic to the Dedekind completion of $\fQ^*$ equipped with the cyclic order deduced from $\leqslant.$
\end{lemma}

\begin{proof}
We only prove the first statement. The second can be shown using a similar argument.
Let $A\subset \fQ^*$ be a Dedekind cut.
Observe that the topological closure of $\iota^+(A)$ inside $\fC$ is a closed interval of the form $[o,x]$ for some $x\in \fC^*.$
Write $j:\fC^*\onto \fR$ for the quotient map. 
We claim that $f:A\mapsto j(x)$ defines an isomorphism $(\Ded(\fQ^*),\leqslant)\to (\fR,\leqslant)$ that produces a commutative diagram $f\circ C=j\circ \iota^+$ for the canonical inclusion $C:\fQ^*\into \Ded(\fQ^*),
x\mapsto(o,x)$
and for $j\circ\iota^+:\fQ^*\into \fR$.
It is immediate that $f$ is order-preserving.
Now, if $q\in \fQ^*$, then it defines the Dedekind cut $C(q):=A_q=(o,q)$ so that $\overline{\iota^+(A_q)}=[o,q^-]$ and thus $j(A_q)=j\circ \iota^+(q)$.
It remains to show that $f$ is bijective.
Consider $g:\fR\to \Ded(\fQ^*)$ defined as $g(x)=\{ q\in \fQ^*:\ j\circ\iota^+(q)<x\}.$
One can check that $g$ is the inverse of $f$ implying that $f$ is indeed bijective.
\end{proof}

\begin{theorem}\label{theo:rigid-FS}
The action $G\act \fC$ (resp.~$G^T\act \fC$) factorises into an order-preserving action $G\act (\fR,\leqslant)$ (resp.~$G^T\act (\fR^T,Cr)$).

Assume now that the FS category $\cF$ is \emph{simple} (which implies $\al:G^V\act\fC$ is faithful). Write $G^0\lhd G$ for the normal subgroup of $g\in G$ acting trivially around $o$ and $\omega$.

We have that:
\begin{itemize}
\item $D(G^0)\act (\fR,\leqslant)$ is O-rigid; 
\item $D(G^T)\act (\fR^T,Cr)$ is C-rigid; and
\item $D(G^V)\act (\fR^V,\tau)$ is T-rigid.
\end{itemize}
\end{theorem}

\begin{proof}
We prove the $F$-case first.
By construction $(\fR,\leqslant)$ is a totally ordered set that is dense and has no extremities.
Since $G\act \fC$ is order-preserving we immediately obtain that it restricts to an action on $\fC^*$ and on the quotient $\fR:=\fC^*/\sim.$
Using the previous lemma we can identity the inclusion $\fQ^*\subset\fR$ with $\fQ^*\subset \Ded(\fQ^*).$
The assumption and Theorem \ref{theo:simplicity} implies that $G\act \fQ^*$ is faithful implying that the action on the Dedekind completion is faithful. Of course, the restriction of the action to $D(G^0)$ continues to be faithful.
By definition of $G^0$, all elements of $G^0$ (and thus of $D(G^0)$) are bounded.
In order to apply Theorem \ref{theo:O-rigid} it is then sufficient to show that $D(G^0)\act (\fQ^*,\leqslant)$ is 2-o-transitive.
This follows easily from Claims 1 and 2 in the proof of Theorem 3.5 in \cite{Brothier-Seelig24}. 
In fact, it is even true that $D(G^0)\act \fQ^*$ is highly-o-transitive.

The $T$-case is rather identical. Let us explain the $V$-case.
We know that $\fR^V=\fC$ is a profinite space. 
Hence, it is compact Hausdorff without isolated points. 
By assumption and Theorem \ref{theo:simplicity} we have that $G^V\act\fR^V$ is faithful and is by construction an action by homeomorphisms.
To apply Theorem \ref{theo:T-rigid} it is sufficient to show that the action of $\Gamma:=D(G^V)$ on $\fR^V$ is locally dense.  
Fix $U\subset \fR^V$ open, a point $x\in U$ and let us show that the closure $\overline{\Gamma_U\cdot x}$ is a neighbourhood of $x$ (where recall $\Gamma_U$ is the subgroup of $g\in\Ga$ supported in $U$).
It is sufficient to show that there exists a cone $x\in C\subset U$ so that for each subcone $D\subset C$ there exists $g\in\Ga_C$ satisfying $g(x)\in D.$
Hence, take any pointed tree $(t,j)$ so that $C:=\Cone(t,j)$, as well as $\Cone(t,j-1)$ and $\Cone(t,j+1)$, are contained inside $U$. Then any subcone $D\subset C$ can be written $\Cone(tf,\ell)$ where $f$ is a forest and $\ell$ is a descendant of $j$. Here we take $f$ so that only its $j$th tree is non-trivial.
Now, the partial transformation $\phi:=\beta(tf,\ell)\circ \beta(t,j)^{-1}$ sends $C$ into $D$.
From there, we then choose any forest $f'$ having only trivial trees except for its $(j-1)$th and $(j+1)$th ones, and is such that that $f$ and $f'$ have same number of leaves.
Take now $\sigma$ a bijection between the leaves of $f$ and $f'$ sending the leaves of $tf'$ equal to the $j$th leaf of $t$ to the leaf $\ell$ of $tf.$ 
Additionally, any leaf of $tf'$ numbered $i$ that is not a descendant of the $(j-1),j$, or $(j+1)$th tree of $f'$ is sent to $i.$ We obtain that $g:=[tf,\sigma,tf']$ is in $G^V$, is supported in $\Cone(t,j-1)\cup\Cone(t,j)\cup\Cone(t,j+1)\subset U$ and restricts to $\phi$ on $\Cone(t,j).$
Now, to transform $g\in G^V$ into an element of $D(G^V)$ we case use a standard construction explain in \cite[Prop.~3.4]{Brothier-Seelig24}.
\end{proof}

\begin{remark}\label{rem:rigid-non-homeo}
\begin{enumerate}
\item Since the rigidity of an action is preserved by taking larger groups we deduce that $G^X\act \fR^X$ is rigid (in the appropriate sense) for $X=F,T,V.$
\item 
The order topologies of $\fR$ and $\fR^T$ are the quotient topologies for the maps $\fC^*\onto \fR$ and $\fC\onto \fR^T.$ When $\fR$ and $\fR^T$ are considered as topological spaces we have that $G\act \fR$ and $G^T\act \fR^T$ are actions by homeomorphisms and are T-rigid. 
\item 
Note that $G^T\act \fR^V$ is not locally dense (which explains why $\fR^V$ is not \emph{the} T-rigid space of $G^T$). Indeed, $G^T_U\cdot q=\{q\}$ for any $q\in \fQ$ and neighbourhood $U$ of $q$ (where recall $G^T_U$ is the subgroup of $g\in G^T$ supported in $U$).
\item 
The space $(\fR,\leqslant)$ is complete dense totally ordered without endpoints. When $\cF$ has countably many colours we get that $(\fR,\leqslant)$ is separable and thus isomorphic to the unit open real interval equipped with its usual order (see for instance \cite[Chap.~6]{Rosenstein82}). 
Similarly, when $\cF$ has countably many colours $(\fR^T,Cr)$ is isomorphic to the circle and $(\fR^V,\leqslant)$ to the Cantor space both equipped with their usual order.
\item We only know one example of an Ore FS category $\cF$ where the triple of spaces $(\fR,\fR^T,\fR^V)$ is not isomorphic to the unit real open interval, the circle, and the Cantor space. 
This happens for example 2.5.2 in \cite{Brothier-Seelig24} where the colour set $S$ is the real interval $(0,1)$ and for all $r,s\in S$ we have the skein relation 
$$Y(s)\circ (Y(r)\ot I) \sim Y(r)\circ (I\ot Y(x)) \text{ where } x=rs+x(1-rs).$$
Then, $G$ is the group of all order-preserving finite piecewise affine homeomorphisms of $(0,1)$. We have $\fR\simeq (0,1)$ and $\fR^T\simeq \R/\Z$, but $\fR^V$ is not second countable.
\item 
The spaces $\fR,\fR^T,\fR^V$ are never pairwise homeomorphic since $\fR$ is non-compact and connected, $\fR^T$ is compact and connected, and $\fR^V$ is compact and totally disconnected. 
Since the rigid spaces are remembered by the groups we deduce that the type $F,T,V$ of an FS group is remembered.
\end{enumerate}
\end{remark}

\subsection{Full groups}\label{sec:full-groups}
Consider a topological space $Z$ and a group acting on it faithfully by homeomorphisms $\Ga\act Z$. 
The \emph{full group of $\Gamma$} is the subgroup $[[\Gamma\act Z]]\subset\Homeo(Z)$ of $h:Z\to Z$ satisfying that there exists a \emph{finite} open covering $\{U_i\}_i$ of $Z$ and some group elements $\{\ga_i\}_i$ in $\Gamma$ so that $h\restriction_{U_i}=\ga_i\restriction_{U_i}$ for all $i.$
Note that certain authors do not require finiteness of the cover. 
The two definitions coincide when $Z$ is compact but otherwise produce two different groups in general (for instance if one considers the standard action of Thompson's group $F\act (0,1)$). The next proposition computes various full groups.

\begin{proposition}\label{prop:local-completion}
Let $\cF$ be an Ore FS category with FS groups $G^F$, $G^T$, $G^V$ and faithful canonical action $\al:G^V\act\fC$. Let $G^0\lhd G$ be the normal subgroup whose elements fix a neighbourhood of the extremities of $\fC$ under $\al:G^V\act\fC$. Let $\ga^X:G^X\act\fR^X$ be the rigid action of $G^X$ as constructed in the previous section.
We have the following equalities:
\begin{enumerate}
    \item $[[D(G^0)\act \fR^F]]=\gamma^F(G^0)$;
    \item $[[G^F\act\fR^F]]=\gamma^F(G^F)$;
    \item $[[ D(G^T)\act\fR^T]] = \gamma^T(G^T)$;
    \item $[[ D(G^T)\act\fR^V]] = \gamma^V(G^V)$
\end{enumerate}
In particular, item (4) implies $[[D(G^V)\act \fR^V]]=\gamma^V(G^V).$
\end{proposition}

\begin{proof}
Let $\cF$ be an Ore FS category as above and let $\cT$ be its set of trees. We will start by proving item (4) as the others will follow from it.
\vspace{2mm}

{\it Proof of item (4).} Write $H^V$ for $[[D(G^T)\act\fR^V]]$ and since the canonical action of $G^V$ is faithful we may identify $G^V$ with the subgroup $\gamma^V(G^V)\subset\Homeo(\fR^V).$
As usual, we identify $G^T$ with a subgroup of $G^V.$
Fix $h\in H^V$ and let us show that $h\in G^V$. 
By definition of full groups there exists a finite open cover of $\fR^V$ so that $h$ coincides with an element of $D(G^T)$ on each open set in the cover.
Each open set can be subdivided into finitely many disjoint cones. 
Up to further subdividing we may assume that $h$ is \emph{FS-affine} on each of the cones (see definition \ref{def:FS-affine}). 
Hence, there exists a finite family of cones $(C_1,\dots, C_n)$ partitioning $\fR^V$ such that $h$ is FS-affine on each $C_k$. 
This means there exist pointed trees $(x_k,i_k)$ and $(y_k,j_k)$ with $1\leqslant k\leqslant n$, such that $h$ restricted to $C_k$ is given by the formula
$$(x_k,i_k)\cdot w\mapsto (y_k,j_k)\cdot w,\ w\in\fR^V.$$

The rest of the proof consists of showing up to even further subdivisions, we may assume that the trees $x_k$ and $y_k$ do not depend on $k$.

Having common right multiples in $\cF$ assures the existence of a tree $s_0\in\cT$ and forests $f_k\in\cF$ satisfying that $s_0=x_k\circ f_k$ for all $1\leqslant k\leqslant n$. 
Let $m$ be the number of leaves of $s_0$ and write $f_{k,l}$ for the $l$th tree of $f_k$. 
For any $1\leqslant\ell\leqslant m$ and any $1\leqslant k\leqslant n$ we have that $h$ restricted to $\Cone(s_0,\ell)$ is the transformation
$$(s_0,\ell)\cdot w\mapsto (y_k,j_k)\circ(f_{k,j_k},\widehat\ell)\cdot w$$
for some leaf $\widehat\ell$ of $f_{k,j_k}$. 
Hence, up to further growing the trees $y_k$, we may assume that there exists a tree $s_0$ with $m$ leaves and pointed trees $(z_1,\nu_1),\dots,(z_m,\nu_m)\in\cT_\ast$ satisfying that $h$ restricted to $\Cone(s_0,\ell)$ is described by
\begin{align}\label{eq:h-formula}
(s_0,\ell)\cdot w\mapsto (z_\ell,\nu_\ell)\cdot w,    
\end{align}
where $1\leqslant \ell\leqslant m$ and $w\in\fR^V$. 
Also, since $h$ is a homeomorphism the family of cones $(\Cone(z_\ell,\nu_\ell):1\leqslant\ell\leqslant m)$ is
\begin{enumerate}
    \item\label{eq:disjoint} pairwise disjoint; and
    \item\label{eq:cover} covers $\fR^V$.
\end{enumerate}
We have made the $x_k$'s a fixed tree $s_0$. We will now work on the $z_\ell$'s.
Again, since $\cF$ has common right multiples, there exists a tree $t\in\cT$ and forests $g_\ell\in\cF$ such that $$t=z_\ell\circ g_\ell \text{ for all } 1\leqslant \ell\leqslant m.$$
Write $M$ for the number of leaves of $t$. We write $g_{\ell,\nu_\ell}$ for the $\nu_\ell$th tree of $g_\ell$. Let us consider the forest
\begin{align*}
g:=g_{1,\nu_1}\ot g_{2,\nu_2}\ot\cdots\ot g_{m,\nu_m},
\end{align*}
set $s:=s_0\circ g$, and write $N$ for the number of leaves of $s$.

\vspace{2mm}
{\it Claim 1: There exists a bijection $\pi:\Leaf(s)\to\Leaf(t)$, hence $N=M$.}
\vspace{2mm}

Let $1\leqslant p\leqslant N$ be a leaf of $s=s_0\circ g$. There are unique $1\leqslant\ell\leqslant m$ and $\widehat p\in\Leaf(g_{\ell,\nu_{\ell}})$ such that $p$ corresponds to the $\widehat p$th leaf of the $\ell$th tree in $g$. Let $1\leqslant\pi(p)\leqslant M$ be the leaf of $t$ such that $\pi(p)$ corresponds to the $\widehat p$th leaf of the tree $g_{\ell,\nu_{\ell}}$ inside $t=z_{\ell}\circ g_{\ell}$. This furnishes a well-defined function $\pi:\Leaf(s)\to\Leaf(t)$ that we now show is bijective. Suppose $1\leqslant p,q\leqslant N$ are distinct leaves of $s$. If $p$ and $q$ are on different trees of $g$, then $\pi(p)$ and $\pi(q)$ are leaves of $t$ that are descendants of pointed trees whose cones are disjoint by (\ref{eq:disjoint}). Hence, $\pi(p)\not=\pi(q)$. If $p$ and $q$ are on the same tree in $g$, then their relative position remains unchanged when sent via $\pi$ to leaves of $t$. Hence, if $p\not=q$, then $\pi(p)\not=\pi(q)$. We have shown $\pi$ is injective. Now fix a leaf $1\leqslant p\leqslant M$ of $t$. By (\ref{eq:cover}), $p$ is a descendant of some pointed tree $(z_\ell,\nu_\ell)$ inside $t$, i.e.,~$p$ is on the $\nu_\ell$th tree of $g_\ell$ inside $t=z_\ell\circ g_\ell$. It is now easy to find a leaf of $s$ sent to $p$ under $\pi$. This completes the proof of the claim. 

\vspace{2mm}
Consequently, it makes sense to build the fraction $[t,\pi,s]\in G^V$, which we now show acts identically to $h$.

\vspace{2mm}
{\it Claim 2: We have that $\al([t,\pi,s])=h$.}
\vspace{2mm}

It is enough to show that $\al([t,\pi,s])$ restricted to each $\Cone(s_0,\ell)$ has the formula (\ref{eq:h-formula}). Fix $1\leqslant\ell\leqslant m$ and let $w\in\fR^V$. It follows that $w=(g_{\ell,\nu_\ell},\mu)\cdot\widehat w$ for a unique leaf $\mu$ of $g_{\ell,\nu_\ell}$ and $\widehat w\in\fR^V$. It follows by definition of $\pi$ that $(s_0,\ell)\circ(g_{\ell,\nu_\ell},\mu)\cdot\widehat w$ is sent to $(z_\ell,\nu_\ell)\circ(g_{\ell,\nu_\ell},\mu)\cdot\widehat w$ by $\al([t,\pi,s])$. This holds no matter what leaf of $g_{\ell,\nu_\ell}$ that $w$ passes through. Hence, $\al([t,\pi,s])$ restricted to $\Cone(s_0,\ell)$ is the transformation
\begin{align*}
(s_0,\ell)\cdot w\mapsto(z_\ell,\nu_\ell)\cdot w.
\end{align*}
This is precisely the formula in (\ref{eq:h-formula}). As such, we have $\al([t,\pi,s])=h$.

We have shown that $H^V\subset G^V$. 
To prove the opposite inclusion it is sufficient to show that given two (non-trivial) pointed trees $(t,i), (s,j)$ there exists $g\in D(G^T)$ satisfying that $g\restriction_{\Cone(s,j)}$ is the transformation
$$(s,j)\cdot w\mapsto (t,i)\cdot w.$$
This follows from a standard argument that has been carefully explained in \cite[Prop.~3.4]{Brothier-Seelig24}. 
Hence, $G^V\subset H^V$, and this concludes the proof of item (4).
\vspace{2mm}

{\it Proof of item (3).} To avoid confusions we now stop identifying $G^V$ with $\gamma^V(G^V).$
Consider $H^T:=[[D(G^T)\act\fR^T]]$ and fix $h\in H^T$ where note we are now acting on the rigid space $\fR^T$ for the $T$-type group. 
By definition, $h$ is continuous and locally order-preserving on the connected space $(\fR^T,Cr)$. This implies that $h$ preserves the cyclic order $Cr$. 
Hence, there exists a unique lift $h_0:\fR^V\to \fR^V$ through the canonical quotient $\fR^V\onto\fR^T$ that preserves the cyclic order of $\fR^V$.
Now, $h_0$ acts locally like a transformation of $\gamma^V(D(G^T))$ and thus must be in $H^V$. The proof of above yields $H^V=\gamma^V(G^V)$ and thus $h_0\in \gamma^V(G^V)$.
Since $h_0$ is preserving the cyclic order of $\fR^V$ and is in $\gamma^V(G^V)$ we deduce that $h_0\in \gamma^V(G^T)$ by Proposition \ref{prop:FS-group-order}.
This implies that $h\in \gamma^T(G^T)$ and thus $H^T\subset \gamma^T(G^T).$
A similar argument than in the preceding case permits to obtain the converse inclusion.
\vspace{2mm}

{\it Proof of item (2).} For clarity we drop ``$F$'' superscripts. 
Let $H:=[[G\act \fR]]$ and $h\in H.$
The map $h:\fR\to\fR$ is locally order-preserving and is a homeomorphism. Hence, it is order-preserving and thus extends uniquely into a homeomorphism $h_0$ of $\fR^V\setminus\{o,\omega\}$ where recall $o=\min\fR^V$ and $\omega=\max\fR^V.$
It necessarily satisfy $\lim_o h=o$ and $\lim_\omega h=\omega$.
By setting $h_0(o)=o$ and $h_0(\omega)=\omega$ we obtain a homeomorphism of $\fR^V.$
Since on a \emph{finite} open cover it acts like an element of $G$ it acts in particular like an element of $G^V$ on $\fR^V$.
Therefore, $h_0$ is in $\gamma^V(G^V)$ by the proof of above.
Since it is order-preserving it must be in $\gamma^V(G)$ by proposition \ref{prop:FS-group-order}.
This implies that $h\in\gamma(G)$ and thus $[[G\act\fR]]=\gamma(G).$
\vspace{2mm}

{\it Proof of item (1).} Finally, consider $H^0:=[[D(G^0)\act \fR]]$ and $h\in H^0$.
Since $D(G^0)\subset G$ we deduce that $h\in [[G\act\fR]]=\gamma(G).$
The element $h$ coincides on a finite open cover with elements of $D(G^0)$. Each of these open sets can be chosen to be intervals.
Hence, there exist $x<y$ in $\fR$ and $g,k\in D(G^0)$ so that $hz=gz$ for all $z<x$ and $ht=kt$ for all $t>y$ ($x$ is the right endpoint of the first interval and $y$ is the left endpoint of the last interval).
Since $g,k$ are in $D(G^0)\subset G^0$ there exist $x'<x<y<y'$ so that $gz=z$ for all $z<x'$ and $ht=t$ for all $t>y'.$
We deduce that $h\in G^0$ and thus $H^0\subset G^0.$
The converse is a direct application of \cite[Prop.~3.4.(1)]{Brothier-Seelig24}.
The statement is that if $g,h,k$ are in $G^0$, that $A$ is clopen inside $\fC$ and $C$ is a proper cone of $\fR^V$, then $g\restriction_A=[g,khk^{-1}]\restriction_A$ when are satisfied
$$h(C)\cap C=\varnothing \text{ and } \supp(g)\cup A\subset k(C).$$
Indeed, take $g\in G^0$ that we consider acting on the canonical space $\fC$.
There exists a large enough tree $t$ with $n\geqslant 6$ leaves so that $g$ acts trivially on $\Cone(t,1),\Cone(t,2)$ and $\Cone(t,n-1), \Cone(t,n)$.
Consider $C=\Cone(t,3)$ and $h\in G^0$ satisfying $h(C)\subset \Cone(t,4).$
This can be done by taking $h=[t'',t']$ where $t'$ consists in $t$ to which we glue a caret at the second leaf and $t''$ consists in $t$ to which we glue a caret at the fourth leaf. 
We have that $h$ acts trivially on $\Cone(t,1)$ and on $\Cone(t,5)$, and then $h(\Cone(t,3))\subset \Cone(t,4).$
Take now $A=\cup_{j=3}^{n-2}\Cone(t,j)$.
In particular, $A$ contains $\supp(g).$
We now construct $k\in G^0$ satisfying $A\subset k(C).$
Let $s$ be any tree with $n-4$ leaves.
Glue $s$ on the $(n-1)$th leaf of $t$ to form $s'$ and glue $s$ to the third leaves of $t$ to form $s''.$ Then the element $k:=[s'',s']$ acts trivially on $\Cone(t,1)$ and $\Cone(t,2)$ and sends $A$ to $\Cone(t,3)=:C.$
Therefore, $g$ and the commutator $[g,khk^{-1}]$ coincides on the support of $g$.
From there we deduce that $g$ belongs to the full group of $D(G^0)\act \fC$ and then belongs to the full group of $D(G^0)\act \fR.$ This completes the proof.
\end{proof}

\begin{remark}\label{rem:local-completion}
\begin{enumerate}
\item
If $\Lambda\subset\Gamma$ is subgroup such that $[[\Lambda]]=\Gamma$ then for any intermediate subgroup $\Lambda\subset\Pi\subset\Gamma$ we have $[[\Pi]]=\Gamma$.
\item Suppose the canonical action is faithful so we can identify $G^X$ with $\gamma^X(G^X)$ and let $j:\fC\onto\fR^T$ be the canonical quotient. Since $[[G^T\act\fR^V]]=G^V$, one might think 
$[[G\act\fR^T]]=G^T$, though this is not true. Since $G^T$ acting on $\fR^T$ is full, by item (1) of this Remark we have $[[G\act \fR^T]]\subset G^T$. Now if $g\in[[G\act \fR^T]]$, then $g$ fixes $j(o)$. Moreover, $G$ is characterised as the subgroup of $G^T$ whose elements fix $j(o)$, hence $g\in G$. Thus $[[ G\act\fR^T]]=G$.
\end{enumerate}
\end{remark}

\section{Invariants}\label{sec:invariants}

In this section we use the previous rigidity results to give invariants for the simple derived subgroups of FS groups whose canonical action is faithful. To this end we {\bf fix a presented Ore FS category $\cF=\FS\la S|R\ra$ which we assume has faithful canonical group action.} We write $G^X$ for the $X$-type FS group of $\cF$ and $\gamma^X:G^X\act\fR^X$ for its rigid action as constructed above, $X=F,T,V$. As usual we may drop ``$F$'' superscripts.

\subsection{Abelianisation}

We have the following practical proposition for distinguishing between themselves the simple groups $D(G^T)$ (resp.~the simple groups $D(G^V)$).

\begin{proposition}\label{prop:abelianisation-invariant}
The simple group $D(G^Y)$ ``remembers'' $G^Y_{ab}$ for $Y=T,V$. 
In details, this means that if $\phi:D(G^Y)\to\Lambda$ is a group isomorphism, then $\Lambda$ admits a rigid action $\beta:\Lambda\act \mathfrak L$. 
Then take $[[\Lambda]]$ the full group of this action and consider its abelianisation $[[\Lambda]]_{ab}.$
We have that $\phi$ extends to an isomorphism $G^Y\to [[\Lambda]]$ that descends to an isomorphism $G^Y_{ab}\to [[\Lambda]]_{ab}.$
In particular, if $G$ and $L$ are FS groups satisfying $D(G^Y)\simeq D(L^Y)$, then $G^Y_{ab}\simeq L^Y_{ab}.$
Similarly, $D(G^0)$ remembers $(G^0)_{ab}.$
In particular, if $G,L$ are FS groups satisfying $D(G^0)\simeq D(L^0)$, then $(G^0)_{ab}\simeq (L^0)_{ab}.$
\end{proposition}

\begin{proof}
Take $Y=T$ or $V$.
Since $\cF$ is an FS category with faithful canonical action, Theorem \ref{theo:simplicity} implies $D(G^Y)$ is a simple group. 
By Theorem \ref{theo:rigid-FS} the action $\gamma^Y:D(G^Y)\act\fR^Y$ is rigid.
Proposition \ref{prop:local-completion} implies that the full group of this rigid action is $\gamma^Y(G^Y)\subset\Homeo(\fR^Y).$
Assume that $\phi:D(G^Y)\to \Lambda$ is a group isomorphism.
This automatically implies that $\Lambda$ admits a (unique) rigid action $\beta:\Lambda\act \mathfrak L$ and moreover that there exists a unique homeomorphism $f:\fR^Y\to \mathfrak L$ satisfying $f(g\cdot x)=\phi(g)\cdot f(x)$ for all $x\in \fR^Y$ and $g\in D(G^Y).$
Let $[[\Lambda]]$ be the full group of $\Lambda\act \mathfrak L$.
Conjugating by $f$ yields an isomorphism from $G^Y$ to $[[\Lambda]]$ that extends $\phi.$
It is then standard that $f$ descends into an isomorphism of the abelianisations.

A similar argument implies the second statement using that $D(G^0)\act \fR$ is rigid and that $G^0$ is the full group of $D(G^0)\act \fR$.
\end{proof}

Recall that Thompson group $F$ is the group of (finite) piecewise affine homeomorphisms of $[0,1]$ with slopes being powers of $2$ and breakpoints being dyadic rationals.
One may define a similar group by replacing $\frac{1}{2}$ in the definition by the golden ratio $\tau=\frac{\sqrt 5-1}{2}$.
Greenberg defines such a group (when acting on the real line rather than acting on the unit interval) in \cite[Ex.~1.22(b)]{Greenberg87}. Cleary then rediscovered this group and extensively studied it in \cite{Cleary00} and it is now known as \emph{Cleary's golden ratio slope Thompson group} (or \emph{Cleary's group} for short). Burillo, Nucinkis, and Reeves have shown Cleary's group enjoys calculus of tree-pair diagrams and also defined and studied its $T$ and $V$-versions in \cite{Burillo-Nucinkis-Reeves21, Burillo-Nucinkis-Reeves22}. Since $\tau^2+\tau=1$, there are two ways to split the interval $[0,1]$ into two pieces of length $\tau$ and $\tau^2$. These splittings are parametrised by two colours $a$ and $b$, respectively, and one obtains a skein relation:
\begin{align*}
\ClearyB.
\end{align*}
From this latter point of view we may consider these groups as FS groups and include them in a one parameter family of bicoloured FS groups (called the \emph{generalised Cleary groups}) as in the corollary below. For example, here is the skein relation which comes after the one above:
\begin{align*}
\ClearyThree.
\end{align*}

\begin{corollary}[simplicity and rigidity for the family of generalised Cleary groups]\label{cor:Cleary}
For any $n\geqslant 2$ define the $n$th \emph{Cleary FS category}
$$\cC_n:=\FS\la a,b| \lambda_n(a)=\rho_n(b)\rangle$$
where $\lambda_n$ and $\rho_n$ are left and right vines of length $n$, respectively.
This is a simple Ore FS category with FS groups denoted $G_n^F,G_n^T,G_n^V$.
When $n=2$, we recover the Cleary irrational slope Thompson groups. If $D(G_n^X)\simeq D(G_m^Y)$, then $n=m$, and $X=Y$ for all $n,m\geqslant 2$, $X,Y\in\{T,V\}.$ 
Moreover, these latter groups are infinite simple of type $\mathtt F_\infty$.
\end{corollary}
\begin{proof}
The observation for the $n=2$ case follows from the work of Burillo, Nucinkis, and Reeves; see \cite{Burillo-Nucinkis-Reeves22}.
We have Ore FS categories by Theorem \ref{theo:Ore-categories}. 
The categories are simple by \cite[Cor.~F]{Brothier-Seelig24}. 
By \cite[Theo.~E]{Brothier-Seelig24} we have that the abelianisation of $G_n^Y$ is isomorphic to $\Z/n\Z$ for $n\geqslant 2$, $Y=T,V$.
Applying the previous proposition yields that $D(G_n^Y)$ remembers $n$.
Now, since $D(G_n^Y)$ is T-rigid, it remembers its rigid space which is the circle when $Y=T$ and the Cantor space when $Y=V$.
Hence, $D(G_n^Y)$ remembers $Y$ as well.
Simplicity of the derived subgroups is given by \cite[Cor.~D]{Brothier-Seelig24}. 
Being of type $\mathtt F_\infty$ comes from having finite abelianisation and \cite[Theo.~8.1]{Brothier22}.
\end{proof}

\subsection{Generalities of germ groups for FS groups} 

Fix a faithful action by homeomorphisms $\eta:\Gamma\act Z$. We say $g,h\in\Gamma$ have the same \emph{germ} at $z\in Z$, written $g\sim_zh$, if $g$ and $h$ act the same on a neighbourhood of $z$. 
This defines an equivalence relation and we write $[g]_z$ for the equivalence class of $g$. We write $[\eta:\Gamma\act Z;z]$ for the subgroup $\Gamma_z\subset\Ga$ of elements fixing $z$ modulo $\sim_z$ which forms the \emph{group} of germs of the action $\eta:\Ga\act Z$ at the point $z$ under the multiplication $[g]_z[h]_z=[gh]_z$. If $\eta$ were a rigid action of $\Gamma$, then the list of all germ groups of $\eta$ would be an invariant of the group $\Gamma$, i.e.,~if $\Gamma\simeq\Gamma'$, then $\Gamma'$ admits a rigid action $\eta':\Gamma'\act Z'$ and there is a homeomorphism $\psi:Z\to Z'$ and a group isomorphism $\theta:\Gamma\to\Gamma'$ conjugating the actions and in particular giving $[\eta:\Gamma\act Z;z]\simeq [\eta':\Gamma'\act Z';\psi(z)]$.

\subsubsection{Quotients of $F$-type FS groups}
Recall that $o,\omega$ are the extremities of the totally ordered space $\fC$ and that $\fC=\fR^V$ where $G^V\act\fR^V$ is the rigid action of $G^V$.
Write $G^0_o,G^0_\omega\lhd G$ for the normal subgroups of $G$ of elements fixing a neighbourhood of $o$ and $\omega$, respectively. Since $o$ and $\omega$ are global fixed points for $\al:G\act\fC$, the germ groups at these points are by definition the quotients
\begin{align*}
[G\act\fC;o]=G/G^0_o\quad\textnormal{and}\quad[G\act\fC;\omega]=G/G^0_\omega.
\end{align*}
\textit{We will see examples where $[G\act \fC;\omega]$ is non-amenable; for instance example \ref{ex:germ-group}.}

\subsubsection{Relating germ groups}\label{sec:relating-germs} Since $D(G)\subset G\subset G^T\subset G^V$, we get a tower of inclusions at the level of germ groups for the \emph{canonical action} $\al:G^V\act \fC$ taken at any point $p\in\fC.$ 
If $p\in\fC^*=\fC\setminus\{o,\omega\}$, then we achieve equality. 
To see why, let $g\in G^V$ and pick a cone $C\subset\fC^*$ containing $p$ on which $g$ acts FS-affinely (recall definition \ref{def:FS-affine})
\begin{align*}
g\restriction_C=\beta(t,\pi(\ell))\circ\beta(s,\ell)^{-1}
\end{align*}
for some trees $t,s$, permutation $\pi$, and leaf $\ell$ which is not the first or last leaf of $s$. Since $p$ is not an extremity, up to passing to a subcone of $C$, we can assume $\pi(\ell)$ is not the first or last leaf of $t$. 
It is now not difficult to build an element of $D(G)$ which acts the same as $g$ on $C$ (see proof of item (1) in proposition \ref{prop:local-completion}). 
Hence, the germ of $g$ at $p$ is the same as the germ at $p$ of an element of $D(G)$, so
\begin{align*}
[D(G)\act\fC;p]=[G\act\fC;p]=[D(G^T)\act\fC;p]=[G^V\act\fC;p]
\end{align*}
as long as $p$ is not an extremity of $\fC$.

Recall the map $j:\fR^V\onto\fR^T$ which glues jumps together 
(If $\cF$ has countably many colours, then $j$ is conjugated to the usual map $\{0,1\}^\N\onto \R/\Z,~x\mapsto\sum_{k\geqslant1}\frac{x_k}{2^k}$). 
Fix $p\in j(\fQ).$ We have that $j^{-1}(p)=\{p^-,p^+\}$, where $p^-\in \fQ^-$ and $p^+\in \fQ^+$ and $(p^-,p^+)$ is a jump. 
In particular, certain open interval $(x,y)$ of $p$ in $\fR^T$ can be lifted via $j$ into two neigbourhoods $(x',p^-]$ and $[p^+,y')$ inside $\fR^V$. 
This induces that the germ of $g\in G^T$ at $p$ corresponds to the product of the two germs of $g$ at $p^-$ and at $p^+$. 
We deduce the following which has been carefully proved in \cite[Prop.~2.12]{Brothier-Seelig26}.

\begin{proposition}
The map \begin{align*}
\widehat{j}:[G^T\act\fR^T;p]&\to[G^T\act\fC;{p^-}]\times[G^T\act\fC;{p^+}]\\
[g;p]&\mapsto([g;p^-],[g;p^+])
\end{align*}
is a group isomorphism for each $p\in j(\fQ).$
\end{proposition}

\subsection{Presentations of certain germ groups}\label{sec:pruning}
In this section we show how to calculate germ groups of the canonical action at its endpoints \emph{as long as said action is faithful.}

Recall $\cF=\FS\la S|R\ra$ is a presented simple Ore FS category. For any free tree $t$ (over $S$) let $\prune_o(t)$ (resp.~$\prune_\omega(t)$) be the word of colours read from the root to its \emph{first} (resp.~last) leaf. 
For example 
\begin{align*}
\prune_o(\vcenter{\hbox{\CompleteABC}})=ab\quad\textnormal{and}\quad\prune_\omega(\vcenter{\hbox{\CompleteABC}})=ac.
\end{align*}
The mappings $t\mapsto\prune_\chi(t)$ extend to (non-monoidal) functors
\begin{align*}
\prune_\chi:\FS\la S\ra\to\Mon\la S\ra, \ f\mapsto\prune_\chi(f), \text{ where }\chi\in\{o,\omega\}.
\end{align*}
Via the universal property for presented FS categories (see \cite[Cor.~1.9]{Brothier22}), the functors $\prune_\chi$ factor through
\begin{align*}
\ov\prune_\chi:\cF\to\Mon\la S|\prune_\chi(R)\ra,
\end{align*}
where $\prune_\chi(R)$ denotes the image of the set of skein relations $R$ under $\prune_\chi\times\prune_\chi$. 
Applying the universal enveloping groupoid functor $U$
yields a functor
\begin{align*}
U(\ov\prune_\chi):\Frac(\cF)\to\Gr\la S|\prune_\chi(R)\ra.
\end{align*}
Recall that the functor $U$ transforms (small) categories into groupoids. It is standard that a (category) presentation for a category $\cD$ gives a (groupoid) presentation for $U(\cD)$; see \cite[Chap.~II, Sec.~3.1]{Dehornoy-Digne-Godelle-Krammer-Michel15}. 
Moreover, when the category admits a calculus of right-fractions, $U$ coincides naturally with the functor $\Frac$.

\begin{proposition}[some germ groups from skein presentations]\label{prop:germ-pres}
Let $\cF=\FS\la S|R\ra$ be an Ore FS category with faithful canonical action. The prune maps define two group isomorphisms
$$[G\act\fC ; o]\simeq \Gr\la S|\prune_o(R)\ra \text{ and } [G\act\fC ; \omega]\simeq \Gr\la S|\prune_\omega(R)\ra.$$
\end{proposition}

Before proving the above proposition let us illustrate the results with an example.

\begin{example}\label{ex:germ-group}
Consider the FS category $\cF$ which is generated by two colours $a,b$ and has unique skein relation
\begin{align*}
\SkeinRelation
\end{align*}
This FS category is known to be Ore and has faithful canonical action; \cite[Sec.~2.1.4]{Brothier-Seelig26}. Hence, the previous proposition applies. 
By reading colours on the sides we deduce
$$[G\act \fC;o]\simeq \Gr\langle a,b| aa=b\rangle\simeq \Z \text{ and } [G\act \fC;\omega]\simeq \Gr\langle a,b| aa=bbb\rangle.$$
Note that the germ group at $\omega$ is non-abelian and even non-amenable.
\end{example}

\begin{proof}[Proof of Proposition \ref{prop:germ-pres}]
We prove 3 claims applied to the endpoint $o$. Arguments easily applied to $\omega$ under obvious modifications. 

\vspace{2mm}
{\it Claim 1: The monoids $\cT_o,\cT_\omega\subseteq\cT_*$ of trees pointed at their first and last leaves are Ore monoids.}
\vspace{2mm}

Recall that for all $(t,\ell)\in\cT_*$ the transformation $\beta(t,\ell):\fC\to\fC$ is injective; \cite[Prop.~2.7]{Brothier-Seelig24}. This implies $\beta(\cT_*)$, and hence any of its submonoids, is left-cancellative. In particular, $\beta(\cT_o)$ is left-cancellative. Now let us show $\beta(\cT_o)$ has common-right-multiples. Let $(t,1),(s,1)\in\cT_o$. Since $\cF$ has common right multiples we are able to pick forests $p,q\in\cF$ satisfying $tp=sq$ in $\cF$. Write $p_1$ (resp.~$q_1$) for the first tree of $p$ (resp.~$q$). It follows that $(p_1,1),(q_1,1)\in\cT_o$, and
\begin{align*}
\beta(t,1)\circ\beta(p_1,1)&=\beta(tp,1)
=\beta(sq,1)
=\beta(s,1)\circ\beta(q_1,1).
\end{align*}
Hence, $\beta(\cT_o)$ has common right multiples, and is thus a (right-)Ore monoid. It admits a fraction group $\Frac(\beta(\cT_o))$.

\vspace{2mm}
{\it Claim 2: The function $\kappa_o:S\to\beta(\cT_o),~s\mapsto\beta(Y_s,1)$ 
induces an isomorphism $\kappa_o:\Gr\la S|\prune_o(R)\ra\to\Frac(\beta(\cT_o))$.}
\vspace{2mm}

The universal property of free monoids implies the map $\kappa_o:S\to\beta(\cT_o)$ defined above extends uniquely to a monoid morphism $\Mon\la S\ra\to \beta(\cT_o)$. This morphism is surjective as $\kappa_o(S)=\{\beta(Y_s,1):s\in S\}$ generates $\beta(\cT_o)$.
Furthermore, if $u\sim v$ is a skein relation in $R$, then 
\begin{align*}
\kappa_o(\prune_o(u))=\beta(u,1)
=\beta(v,1)
=\kappa_o(\prune_o(v)).
\end{align*}
By the universal property of quotients, we obtain that $\kappa_o$ factors through a surjection $\Mon\la S|\prune_o(R)\ra\onto\beta(\cT_o)$. Applying the universal enveloping groupoid functor to this monoid morphism gives a surjective group morphism $\kappa_o:\Gr\la S|\prune_o(R)\ra\onto\Frac(\beta(\cT_o))$. Now using the assumption that the canonical action is faithful we will show $\kappa_o:\Gr\la S|\prune_o(R)\ra\onto\Frac(\beta(\cT_o))$ is injective. From above obtain a surjective group morphism $$\prune_o:G\onto\Gr\la S|\prune_o(R)\ra,~[t,s]\mapsto\prune_o(t)\prune_o(s)^{-1}.$$ Hence, any element $\gamma$ in $\Gr\la S|\prune_o(R)\ra$ can be written as $\prune_o([t,s])$ for some fraction of trees $[t,s]\in G$. If $\gamma\in\ker(\kappa_o)$, then $\beta(t,1)\circ\beta(s,1)^{-1}=\id_{\Cone(s,1)}$, which is to say that $\al([t,s])$ acts trivially on $\Cone(s,1)$. 
Since $\cF$ has common right multiples, we pick forests $p,q\in\cF$ such that $tp=sq$ and rewrite $[t,s]=[sq,sp]$. Observe that the first trees $p_1$ and $q_1$ of $p$ and $q$ must have the same number of leaves as otherwise $\al([t,s])$ would act non-trivially on $\Cone(s,1)$. It follows that $[p_1,q_1]$ is a well-defined fraction of trees in $G$ that must satisfy $\al([p_1,q_1])=\id_{\fC}$ again since $\al([t,s])$ acts trivially on $\Cone(s,1)$. 
Since the canonical action of $\cF$ is faithful we have $[p_1,q_1]=e$. This implies $\gamma=\prune_o([t,s])=e$ and so $\kappa_o$ is injective.

\vspace{2mm}
{\it Claim 3: The morphism $G\onto\Frac(\beta(\cT_o)),~[t,s]\mapsto\beta(t,1)\circ\beta(s,1)^{-1}$ factors through an isomorphism on the group of germs $[G\act\fC;o]=G/G^0_o$.} 
\vspace{2mm}

We now identify the group of germs of $G\act\fC$ at $o$ with $\Gr\la S|\prune_o(R)\ra$. Recall $G_o\lhd G$ consists of all group elements acting trivially on a cone containing $o$. Since
\begin{align*}
\kappa_o(\prune_o([t,s]))=\beta(t,1)\circ\beta(s,1)^{-1}
\end{align*}
we obtain that $G_o=\ker(\kappa_o\circ\prune_o)$. Hence, the surjection $\kappa_o\circ\prune_o:G\onto\Frac(\beta(\cT_o))$ factors through an isomorphism $$G/G_o\to\Frac(\beta(\cT_o)),~[t,s]G_o\mapsto\beta(t,1)\circ\beta(s,1)^{-1}.$$
This completes the proof. 
\end{proof}

\section{A class of examples}\label{sec:examples}

In this section we study a class of examples of FS categories and their groups. We begin by considering all Ore FS categories and reduce our generality by making certain assumptions on the skein presentation as we go. 
Our aim is to classify these groups using their canonical dynamic.

\subsection{Dynamics}\label{sec:dynamics} We begin with a fixed presented Ore FS category $\cF=\FS\la S|R\ra$ which we do \emph{not} assume has faithful canonical action. 
Our goal is to find conditions on $\la S|R\ra$ which ensure $\cF$ has faithful canonical action. 
As usual, we write $\cT$ for the tree set and $\cT_\ast$ for the pointed tree monoid of $\cF$. 
We denote the $X$-type FS group of $\cF$ by $G^X$, $X=F,T,V$, and write $\al:G^V\act\fC$ for its canonical group action and $\beta:\cT_\ast\act\fC$ for its canonical monoid action. In this section we will deal with the Cantor space of infinite binary sequences $\{\0,\1\}^\N$. Given a finite binary word $w$ we will write $\ov w$ for the sequence $www\cdots$ which repeats $w$ infinitely many times. Hence, $\ov\0$ and $\ov\1$ are the constant sequences having value $\0$ and value $\1$, respectively. Given a finite binary word $w$ whose length is $n$ and a binary sequence $p$ we may write $w\cdot p$ whose first $n$ entries are those of $w$ and then whose entries are those of $p$.

\subsubsection{Choosing coordinates for the canonical actions}
\label{sec:sequence-space} 

Recall that $\fQ=\cT_\ast/\sim$ is a quotient of the pointed trees having $\fC$ for completion; see section \ref{sec:canonical-actions}.
The inclusion $\iota^+:\fQ\into \fC$ is analogous to the finitely supported binary sequences $\{\0,\1\}^{(\N)}$ inside $\{\0,\1\}^\N$.
Recall that $\{\0,\1\}^*$ is a monoid with word concatenation and $\cT_\ast$ is a monoid with gluing pointed trees. For each colour $a\in S$ let $\varphi_a:\{\0,\1\}^*\to\cT_\ast$ be the unique monoid morphism having
\begin{align*}
\varphi_a(\0)=(Y_a,\0) \ \textnormal{and} \ \varphi_a(\1)=(Y_a,\1).    
\end{align*}
Note that this is injective. Moreover, if $w$ is a word, then
\begin{align*}
\varphi_a(w\cdot\0)=\varphi_a(w)\circ(Y_a,\0)
\end{align*}
whose class is equal to the class of $\varphi_a(w)$ inside $\fQ.$
Hence, $\varphi_a$ factors through an injection $\varphi_a:\{\0,\1\}^{(\N)}\to \fQ$. We note that $\varphi_a$ is order-preserving and continuous. 

\begin{assumption}\label{ass:1}
For the remainder of the section assume the poset of trees $\cT$ in $\cF$ has a \emph{cofinal sequence of $a$-trees.}
\end{assumption}

This means that for any tree $t\in\cT$ there exists an $a$-tree $s$ (a tree with all vertices coloured by $a$) satisfying $t\leqslant s$, or equivalently, $s=t\circ f$ for some forest $f$.
It follows that any pointed tree can be grown into a pointed $a$-tree. 
Note that the FS categories from the introduction of type (\ref{eq:skein-presC}) do \emph{not} have a cofinal sequence of $a$-trees, whereas the categories of type (\ref{eq:skein-presA}) and (\ref{eq:skein-presB}) do. 
As a consequence of having a cofinal sequence of $a$-trees we have that $\varphi_a:\{\0,\1\}^{(\N)}\to \fQ$ is surjective and extends uniquely to an order-preserving homeomorphism  $\varphi_a:\{\0,\1\}^\N\to\fC$. 
We obtain actions $\widehat{\beta}:\cT_*\act\{\0,\1\}^\N$ and $\widehat{\al}:G^V\act\{\0,\1\}^\N$ through conjugating the actions $\beta$ and $\al$ by $\varphi_a$, that is
\begin{align*}
\widehat{\beta}(t,i)(p):=\varphi_a\circ\beta(t,i)\circ\varphi_a^{-1}(p)\quad\textnormal{and}\quad\widehat{\al}(g)( p):=\varphi_a\circ\al(g)\circ\varphi_a^{-1}(p)
\end{align*}
for all pointed trees $(t,i)\in\cT_*$, fractions $g\in G^V$, and sequences $p\in\{\0,\1\}^\N$. We are now focused on understanding the actions $\widehat{\beta}$ and $\widehat{\al}$ and we henceforth refer to these as the canonical actions of $\cF$. Note that $\widehat{\al}$ is determined by $\widehat{\beta}$ in a similar way as described in section \ref{sec:topology-and-compatibility}. For each colour $b\in S$ it will be convenient to notate $$B_{\0}:=\widehat{\beta}(Y_b,1) \text{ and } B_{\1}=\widehat{\beta}(Y_b,2).$$ Here the choice of upper case letter corresponds to the lower case colour, for instance $C_\0=\widehat{\beta}(Y_c,1)$ for some other colour $c\in S$. Our choice to use ``$a$-coordinates'' via $\varphi_a:\{\0,\1\}^\N\to\fC$ implies the actions of the pointed $a$-carets under $\widehat{\beta}$ is as simple as possible:
\begin{align*}
A_\0(p)=\0\cdot p\quad\textnormal{and}\quad
A_\1(p)=\1\cdot p \quad \text{ for all } p\in\{\0,\1\}^\N.
\end{align*}
The action for the other colours is more complicated.

\subsubsection{The action of pointed carets}\label{sec:pointed-carets} 
The remainder of this section is dedicated to understanding the action of the pointed carets of colours different from $a$. 
Recall that for any finite uncoloured binary tree $t$ there is an order-preserving map
\begin{align*}
\ell_t:\{1,2,\dots,n\}\to\{\0,\1\}^*,~i\mapsto\ell_t(i)
\end{align*}
which gives the address of the $i$th leaf of $t$. 
This extends naturally to a map $\ell_t:\{1,2,\dots,n\}^\N\to\{\0,\1\}^\N$ by $$\ell_t(i_1i_2i_3\dots):=\ell_t(i_1)\ell_t(i_2)\ell_t(i_3)\cdots.$$
This map is manifestly injective and continuous. 
Moreover, since the addresses of the leaves of $t$ partition $\{\0,\1\}^\N$ it follows that $\ell_t$ is also surjective. It follows that $\ell_t$ is a homeomorphism. Heuristically, the map $\ell_t$ consists of identifying the rooted infinite binary tree with the rooted infinite $n$-ary tree in which each $n$-ary caret is replaced by a copy of the tree $t$. Now, let us fix two finite uncoloured binary trees $x$ and $y$ both having $n>1$ leaves. 

\begin{assumption}\label{ass:2}
For the remainder of the section we assume $\cF$ has skein presentation $\la a,b|x(a)=y(b)\ra$.   
\end{assumption}

This is coherent with our previous assumption of having a cofinal sequence of $a$-trees. The skein relation $x(a)=y(b)$ implies $$\varphi_a\circ\ell_x=\varphi_b\circ\ell_y$$ as maps $\{1,2,\dots,n\}^\N\to\fC$. Recall here that $\varphi_c:\{\0,\1\}^\N\to\fC$ is the $c$-vertex embedding for $c\in\{a,b\}$.
With these considerations we can refine our understanding of $B_\0$ and $B_\1$. From the definitions we unwrap
\begin{align*}
B_{\tt i}&=\widehat{\beta}(Y_b,{\tt i})\\
&=\varphi_a\circ\beta(Y_a,{\tt i})\circ\varphi_a^{-1}\\
&=\ell_x\circ\ell_y^{-1}\circ\varphi_b\circ\beta(Y_b,{\tt i})\circ\varphi_b^{-1}\circ\ell_y\circ\ell_x^{-1}
\end{align*}
for both ${\tt i}\in\{\0,\1\}$. Let us explain in words what this transformation does to a binary sequence $p\in\{\0,\1\}^\N$.
\begin{enumerate}
    \item First, $p$ is uniquely identified with a sequence whose entries are leaves of the binary tree $x$.
    
    For example if the address of the second leaf of the tree $x$ is $\1\0$ and $p=\1\0\1\0\1\0\cdots$, then we can write $p=\ell_x(2)\ell_x(2)\ell_x(2)\cdots$.
    \item The coordinate change homeomorphism $\ell_y\circ\ell_x^{-1}$ replaces the $i$th entry $\ell_x(k_i)$ of $p$ with $\ell_y(k_i)$ for all $i\geqslant1$. 
    This gives a binary sequence $q=\ell_y\ell_x^{-1}(p)$.
    
    Continuing the example above, suppose the second leaf of the tree $y$ had address $\0\1$. Then $\ell_y\ell_x^{-1}(\1\0\1\0\1\0\cdots)=\ell_y(222\cdots)=\0\1\0\1\0\1\cdots$.
    \item The action $\varphi_b\circ\beta(Y_b,{\tt i})\circ\varphi_b^{-1}$ is done using only the colour $b$. It is then easy to describe as: $\varphi_b\circ\beta(Y_b,{\tt i})\circ\varphi_b^{-1}(q)={\tt i}\cdot q$ for ${\tt i}\in\{\0,\1\}$. 
    
    In the example above we have that $\varphi_b\circ\beta(Y_b,{\tt i})\circ\varphi_b^{-1}(\0\1\0\1\0\1\cdots)={\tt i}\0\1\0\1\0\1\cdots$ because the sequence $\0\1\0\1\0\1\cdots$ is in ``$b$-coordinates''.
    \item Finally, ${\tt i}\cdot q$ is uniquely identified with a sequence whose entries are leaves of the binary tree $y$. 
    Applying $\ell_x\circ\ell_y^{-1}$ consists in interpreting a pointed $b$-tree (or a limit of those) into a pointed $a$-tree. 
    A concrete description of this transformation is difficult to obtain without further assumptions on $x$ and $y$. 
\end{enumerate}

Recall that a right-vine is a tree define by starting with a caret and repeatedly gluing carets to the last leaf of the previous tree. 

\begin{assumption}\label{ass:3}
For the remainder of the section assume that the tree $y$ in the skein presentation $\la a,b|x(a)=y(b)\ra$ is a \emph{right-vine}. To emphasise this assumption we now write $\rho$ instead of $y$ ($\rho$ for \emph{right}-vine). 
\end{assumption}

Hence, the addresses of the leaves of $\rho$ are given by
\begin{align}\label{eq:right-vine-leaves}
\ell_\rho(i)=\1^{i-1}\0 \quad\textnormal{and}\quad \ell_\rho(n)=\1^{n-1}
\end{align}
for $1\leqslant i<n$ (compare to the diagrams \ref{eq:vines}). We introduce some more notation for the tree $x$. We write $L_x$ (resp.~$R_x$) for the length of the left (resp.~right) side of $x$, that is, $\ell_x(1)=\0^{L_x}$ and $\ell_x(n)=\1^{R_x}$. Since $\rho$ is a right-vine we have a prune relation $$(Y_a,\0)^{L_x}=\prune(x(a),\0^{L_x})\sim_{prune}\prune(\rho(b),\0)=(Y_b,\0).$$ 
Hence, from section \ref{sec:pruned} we obtain that the pointed trees $(Y_a,\0)^{L_x}$ and $(Y_b,\0)$ act the same under the canonical action $\beta$. Hence, $A_\0^{L_x}=B_\0$, and so
\begin{align*}
B_\0(p)=\0^{L_x}\cdot p \text{ for all } p\in\{\0,\1\}^\N
\end{align*}
We now understand well the action of $B_\0$ with the assumptions we have made.
\emph{It remains to understand the action of $B_\1$.}

\subsubsection{An infinite tree made from the finite tree $x$.}\label{sec:infinite-tree} 
To understand the action of $B_\1$ we now partition the Cantor space into infinitely many cones which are adapted to $B_\1$. 
First, recall that a \emph{quasi-tree} with \emph{cell} 
$x$ ($x$ being a tree) is a tree built by inductively gluing together copies of $x$; see \cite[Sec.~1.1.1]{Brothier-Seelig24}. 
Define $T_x$ to be the \emph{infinite} quasi-right-vine having cell $x$. In other words $T_x$ is constructed by starting with $x$ and inductively gluing another copy of $x$ to the last leaf. For example
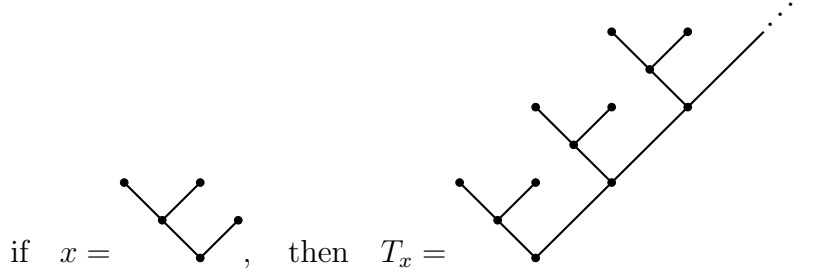
\begin{align*}
\textnormal{if}\quad x=\Cell,\quad\textnormal{then}\quad T_x=\InfiniteQuasiRightVine.
\end{align*}
The address of the $i$th leaf of $T_x$ is thus the binary word $w_i:=\ell_x(n)^j\cdot \ell_x(k)$,
where $i=j(n-1)+k$, $j\geqslant0$, and $1\leqslant k\leqslant n-1$. Here $n$ is the number of leaves of $x$. 
Using this quasi-tree we can understand the action of $B_\1$. 
Recall a sequence $p\in\{\0,\1\}^\N$ has \emph{tail} $q\in\{\0,\1\}^\N$ if $p=w\cdot q$ for some finite word $w\in\{\0,\1\}^*$. Having the same tail is an equivalence relation on $\{\0,\1\}^\N$.

\begin{proposition}\label{prop:infinite-tree}
For all $i\geqslant1$ and $q\in\{\0,\1\}^\N$ we have $B_\1(w_i\cdot q)=w_{i+1}\cdot q$. In particular, $B_\1$ preserves tail equivalence.
\end{proposition}

\begin{proof}
Recall that $B_\1:=\widehat{\beta}(Y_b,\1)$. Since $\rho$ is a right-vine the addresses of the leaves of $\rho$ satisfy
\begin{align}\label{eq:right-vine}
\1\cdot\ell_\rho(k)=\ell_\rho(k+1)\quad\textnormal{and}\quad\1\cdot\ell_\rho(n-1)=\ell_\rho(n)\cdot\ell_\rho(1)
\end{align}
for all $1\leqslant k<n-1$; see (\ref{eq:right-vine-leaves}). Let $i\geqslant1$ and consider the address of the $i$th leaf $w_i$ of the infinite binary tree $T_x$.
We have
\begin{align*}
\ell_\rho\ell_x^{-1}(w_i\cdot q)=\ell_\rho(n)^j\cdot\ell_\rho(k)\cdot \ell_\rho\ell_x^{-1}(q)
\end{align*}
where $i=j(n-1)+k$. Applying $\varphi_b\circ\beta(Y_b,\1)\circ\varphi_b^{-1}$ to this gives
\begin{align*}
\varphi_b\circ\beta(Y_b,\1)\circ\varphi_b^{-1}(\ell_\rho(n)^j\cdot\ell_\rho(k)\cdot q)&=\ell_\rho(n)^j\cdot(\varphi_b\circ\beta(Y_b,\1)\circ\varphi_b^{-1})(\ell_\rho(k)\cdot\ell_\rho\ell_x^{-1}(q))\\
&=\ell_\rho(n)^j\cdot\1\cdot\ell_\rho(k)\cdot\ell_\rho\ell_x^{-1}(q).
\end{align*}
Using (\ref{eq:right-vine}) we deduce this is $\ell_\rho(n)^j\cdot\ell_\rho(k+1)\cdot\ell_\rho\ell_x^{-1}(q)$ if $k<n-1$ and $\ell_\rho(n)^{j+1}\cdot\ell_\rho(1)\cdot\ell_\rho\ell_x^{-1}(q)$ if $k=n-1$.
Finally, applying $\ell_x\ell_\rho^{-1}$ gives
\begin{align*}
B_\1(w_i\cdot q)=
\begin{cases}
\ell_x(n)^j\cdot\ell_x(k+1)\cdot q&\textnormal{if $k<n-1$}\\
\ell_x(n)^{j+1}\cdot\ell_x(1)\cdot q&\textnormal{if $k=n-1$}.
\end{cases}
\end{align*}
In either case this is just $w_{i+1}\cdot q$.
\end{proof}

\begin{remark}\label{rem:B1}
\begin{enumerate}
    \item The family of cones $(\{w_i\cdot p:p\in\{\0,\1\}^\N\}:i\geqslant1)$ partitions $\{\0,\1\}^\N\setminus\{\ov\1\}$. The above proposition asserts that around every point of $\{\0,\1\}^\N\setminus\{\ov\1\}$ we can find a neighbourhood on which $B_\1$ acts like a pointed $a$-tree. In fact, $B_\1$ is locally ``affine'' on $\{\0,\1\}^\N\setminus\{\ov\1\}$. Moreover, $B_\1$ is order-preserving and so it fixes $\ov\1$. As we now have an explicit description of the maps $A_\0$, $A_\1$, $B_\0$, and $B_\1$, we can completely describe the canonical monoid action $\widehat{\beta}:\cT_\ast\act\{\0,\1\}^\N$ and canonical group action $\widehat{\al}:G^V\act\{\0,\1\}^\N$ in the setting of assumptions \ref{ass:1}, \ref{ass:2}, and \ref{ass:3}. We leave pushing beyond these assumptions for future work.   
    \item That $B_\1$ preserves tail equivalence is not true \emph{in general}. Indeed, for the skein presentation $\la a,b|a_1a_1a_2=b_1b_2b_2\ra$ one may quickly verify that $B_\1(\ov{\0\1\1})=\ov{\1\0\0}$.
    \item By construction, the restriction of the action $\widehat{\al}:G^V\act\{\0,\1\}^\N$ to the $a$-copy of Thompson's group $V$ inside $G^V$ is the usual action of $V$ on binary sequences. 
This can be explicitly verified using the descriptions of the caret maps $A_\0,A_\1$ of above. Let $j:\{\0,\1\}^\N\onto\R/\Z$ be the usual surjection
\begin{align*}
j(p)=\sum_{k\geqslant1}\frac{p_k}{2^k}.
\end{align*}
The canonical action $\widehat{\al}:G^V\act\{\0,\1\}^\N$ induces actions $\gamma^T:G^T\act \R/\Z$ and $\gamma:G\act(0,1)$ through the map $j$. 
By the above observation, these actions restricted to $T$ and $F$ are just their standard actions, which are rigid. It follows that if the actions $\gamma^T:G^T\act\R/\Z$ and $\gamma:G\act(0,1)$ of the larger groups $G^T$ and $G$ are faithful, then they are automatically rigid.
This corroborates our general results from section \ref{sec:rigid-FS} in the special case of assumption \ref{ass:1}.
\end{enumerate}
\end{remark}

\subsubsection{Computation of germ groups} 
Under our current assumption we are able to compute the missing germ groups of $G\act \fC$ from proposition \ref{prop:germ-pres}.

\begin{proposition}\label{prop:germ-not-singular}
Let $\cF=\FS\la a,b|x(a)=\rho(b)\ra$ have canonical actions $\widehat{\al}:G^V\act\{\0,\1\}^\N$ and $\widehat{\beta}:\cT_\ast\act\{\0,\1\}^\N$ as described above. Suppose $q\in\{\0,\1\}^\N$ does \emph{not} have tail $\ov\1$.
\begin{enumerate}
    \item For all $(t,\ell)\in\cT_\ast$ there is prefix $u$ of $q$ and a binary word $v$ such that $\beta(t,\ell)(u\cdot z)=v\cdot z$ for all $z\in\{\0,\1\}^\N$. 
    \item The $a$-vertex embedding $V\into G^V$ induces an isomorphism of germ groups $[V\act\{\0,\1\}^\N;q]\simeq [G^V\act\{\0,\1\}^\N;q]$.
\end{enumerate}
\end{proposition}
\begin{proof}
{\it Proof of item (1).} Fix a sequence $q\in\{\0,\1\}^\N$ whose tail is not $\ov\1$ meaning $q$ is \emph{not} of the form $w\cdot\ov\1$ for some binary word $w$. 
Note the set $\Pi:=\{A_\0,A_\1,B_\0,B_\1\}$ generates the image monoid $\widehat{\beta}(\cT_\ast)$. 
We will prove the claim by induction on length of a product of generators from $\Pi$ representing $\widehat{\beta}(t,\ell)$. Among length $1$ words the claim is trivial for the generators $A_\0$, $A_\1$, and $B_\0$. Indeed, we may pick the prefix $u$ of $q=\varnothing\cdot q$ to be the empty word $\varnothing$ as $$A_\0(\varnothing\cdot z)=\0\cdot z,~A_\1(\varnothing\cdot z)=\1\cdot z,~\textnormal{and}~B_\0(\varnothing\cdot z)=\0^{L_x}\cdot z$$ for all $z\in\{\0,\1\}^\N$. 
Here the choice of $v$ is $\0$, $\1$, and $\0^{L_x}$, respectively. 
As for $B_\1$, since $q$ does not have tail $\ov\1$, we are guaranteed that it passes through some leaf $w_i$ of the infinite tree $T_x$ from section \ref{sec:infinite-tree}. Hence, if we choose $u=w_{i}$, which is a prefix of $q$ as just explained, and $v=w_{i+1}$, we have that $$B_\1(u\cdot z)=B_1(w_i\cdot z)=w_{i+1}\cdot z=v\cdot z$$ for all $z\in\{\0,\1\}^\N$. We have just proven the claim holds for length $1$ words in the generators $\Pi$. Assume now the claim holds for length $k\geqslant2$ words over $\Pi$ and let $\widehat{\beta}(t,\ell)$ be represented by a word of length $k+1$. We may factor $\widehat{\beta}(t,\ell)$ as $Z\circ W$, where $Z\in\Pi$ is a pointed caret action, and $W$ has length $k$. By the inductive hypothesis pick a prefix $u$ of $q$ and a binary word $v$ such that $W(u\cdot z)=v\cdot z$ for all $z\in\{\0,\1\}^\N$. Similar to the base case, if $Z$ is $A_\0$, $A_\1$, or $B_\0$, then the claim is clear as $$A_\0W(u\cdot z)=\0v\cdot z,~A_\1W(u\cdot z)=\1v\cdot z,~\textnormal{and}~B_\0W(u\cdot z)=\0^{L_x}v\cdot z.$$
It remains to consider $Z=B_\1$. Since $u$ is a prefix of $q$ we may write $q=u\cdot q'$ for some $q'\in\{\0,\1\}^\N$. Since $q$ does not have tail $\ov\1$ we may pick long enough prefix $u'$ of $q'$ such that the word $v\cdot u'$ passes through one of the leaves of the infinite binary tree $T_x$. Hence, we may write $v\cdot u'=w_j\cdot u''$ where $w_j$ is a leaf of $T_x$ and $u''$ is some (possibly empty) binary word. We now have
\begin{align*}
\widehat{\beta}(t,\ell)(u\cdot u'\cdot z)=B_\1W(u\cdot u'\cdot z)=B_\1(v\cdot u'\cdot z)
=B_\1(w_j\cdot u''\cdot z)
=w_{j+1}\cdot u''\cdot z
\end{align*}
for all $z\in\{\0,\1\}^\N$. Since $u\cdot u'$ is a prefix of $q$, this proves the claim.

\vspace{2mm}
{\it Proof of item (2).} Again, let $q\in\{\0,\1\}^\N$ be a sequence whose tail is not equal to $\ov\1$. The $a$-vertex embedding $V\into G^V$ induces an embedding on germ groups $[V\act\{\0,\1\}^\N;q]\into[G^V\act\{\0,\1\}^\N;q]$. We aim to show this is surjective. To this end, let $g\in G^V$ and pick a triple $(t,\pi,s)$ representing $g$ such that $s$ is $a$-coloured (possible since $\cF$ has a cofinal sequence of $a$-trees). Let $\ell$ be the address of the leaf of $s$ whose cone contains the sequence $q$ so that $q=\ell\cdot y$ for some $y\in\{\0,\1\}^\N$. Since $q$ does not have tail $\ov\1$ it follows that $y$ does not have tail $\ov\1$. With this we may use item (1) to pick a finite prefix $u$ of $y$ and a binary word $v$ such that $\widehat{\beta}(t,\pi(\ell))(u\cdot z)=v\cdot z$ for all $z\in\{\0,\1\}^\N$. By compatibility of the canonical monoid and group actions we have
\begin{align*}
\widehat{\al}(g)(\ell\cdot u\cdot z)&=\widehat{\beta}(t,\pi(\ell))\circ\widehat{\beta}(s,\ell)^{-1}(\ell\cdot u\cdot z)
=\widehat{\beta}(t,\pi(\ell))( u\cdot z)
=v\cdot z
\end{align*}
for all $z\in\{\0,\1\}^\N$. It is not hard to construct an element $h\in V$ that acts by prefix replacement $\ell\cdot u\cdot z\mapsto v\cdot z$ on the cone determined by $\ell\cdot u$. Hence, the germ of $h\in V$ at $q$ is that of the germ of $g\in G^V$ at $q$, and so the induced map $[V\act\{\0,\1\}^\N;q]\into[G^V\act\{\0,\1\}^\N;q]$ is surjective.
\end{proof}

\subsection{Classification}\label{sec:isomorphism-problem} We now prove Theorem \ref{mainthm:B}.
Consider the FS categories $\cH_n,\cG_t,\cJ_t$ from the introduction with $n\geqslant 2$ and $t$ a non-trivial uncoloured tree.

\subsubsection{The groups are simple}
The category $\cH_n$ was proven to be Ore and have faithful canonical action in \cite[Sec.~7]{Brothier-Seelig24}.
Hence, it produces three groups $H_n,H_n^T,H_n^V$ so that $D(H_n^T)$ and $D(H_n^V)$ are simple.
Additionally, $H_n$ was proven to be isomorphic to the $n$-ary Thompson group (or $F$-type Higman--Thompson group).
In particular, $D(H_n)=D(H_n^0)$ and this group is simple but not finitely generated by Proposition \ref{prop:not-fg}.
The categories $\cG_t$ and $\cJ_t$ are Ore by \cite[Theo.~8.3]{Brothier22}. 
The category $\cG_t$ was proven to have faithful canonical action in \cite[Sec.~6]{Brothier-Seelig24}. 
A slight generalisation of \cite[Theo.~2.6]{Brothier-Seelig26} implies $\cJ_t$ has faithful canonical action.
They provide six groups $G_t,G_t^T,G_t^V,J_t,J_t^T,J_t^V$ that are all finitely presented (and are in fact of type $\tt F_\infty$) by \cite[Theo.~8.3]{Brothier22}.
We consider the subgroups $G_t^0\subset G_t$ and $J_t^0\subset J_t$ obtained by taking the elements acting trivially near the extremities of $\fC$ in their canonical actions.
The groups $D(G_t^0),D(J_t^0)$ are simple but are not finitely generated by Proposition \ref{prop:not-fg}.
We will later show that $D(G_t^0)=D(G_t)$ and $D(J_t^0)\neq D(J_t).$
The other derived groups $D(G_t^T),D(G_t^V),D(J_t^T),D(J_t^V)$ are simple since their FS categories have faithful canonical action.

\subsubsection{Classification using abelianisation}
Write $\cE_F,\cE_T,\cE_V$ for the families of the $D(H_n)$, $D(G_t^0)$, $D(J_t^0)$, the $D(H_n^T)$, $D(G_t^T)$, $D(J_t^T)$, and the $D(H_n^V)$, $D(G_t^V)$, $D(J_t^V)$, respectively.
We have proven that the groups of $\cE:=\cE_F\cup \cE_T\cup \cE_V$ are all simple.
Moreover, Theorem \ref{theo:rigid-FS} implies that they are all rigid.
The rigid spaces of groups in $\cE_F,\cE_T,\cE_V$ are homeomorphic to the open real unit interval, the circle, and the Cantor space, respectively.
By uniqueness of the rigid space we deduce that if $A,B\in\cE$ and are isomorphic, then they must both be in $\cE_F,\cE_T$, or in $\cE_V.$
We now compute the abelianisation of the $H_n^X,G_t^X,J_t^X$ where $X=T$ or $V$.
Theorem \ref{theo:abelianisation} implies that
\begin{align*}
(G_t^X)_{ab}\simeq\Z/N_t\Z,\quad(J_t^X)_{ab}\simeq\Z/(N_t+1)\Z,\quad(H_n^X)_{ab}\simeq\Z^{n-2}.
\end{align*}
where $N_t$ denotes the number of leaves in the tree $t$.
Since $t$ is assumed to be a non-trivial tree it has at least two leaves and thus $N_t\geqslant 2$ and $N_t+1\geqslant 3$.
In particular, $D(G_t^X)$ and $D(J_t^X)$ are proper finite index subgroups of $G_t^X$ and $J_t^X$. 
They are finitely presented (in fact of type $\tt F_\infty$) since $G_t^X$ and $J_t^X$ are.
By proposition \ref{prop:abelianisation-invariant}, we deduce that 
$$D(G_t^X)\simeq D(G_s^X)\Rightarrow t,s \text{ have same number of leaves}$$
and a similar statement holds for the $J_t^X$'s.
Additionally, $$D(H_n^X)\simeq D(H_m^X)\Rightarrow n=m.$$
Moreover, $D(H_n^X)$ is not isomorphic to $D(G_t^X)$, nor $D(J_t^X)$.

Consider now the $F$-case with the groups $D(H_n^0),D(G_t^0),D(J_t^0).$
The abelianisations of $H_n^0,G_t^0,J_t^0$ are remembered by their derived subgroup (see Proposition \ref{prop:abelianisation-invariant}). 
Moreover, these abelianisations are isomorphic to the abelianisation of the corresponding $T$-type FS groups by Corollary \ref{cor:abelianisation-Gzero}.
We deduce similar conclusion that in the $T$ and $V$-cases: $D(G^0_t)\simeq D(G^0_s)$ (resp.~$D(J_t^0)\simeq D(J_s^0)$) implies $N_t=N_s$, $D(H_n)\simeq D(H_m)$ implies $n=m$, and $D(H_n)\not\simeq D(G^0_t),D(J^0_t).$

\subsubsection{Classification using germ groups}

It remains to show that $D(G_t^X)\not\simeq D(J_s^X)$ for $X=T,V$ and that $D(G_t^0)\not\simeq D(J_s^0).$
We do it by comparing the list of all germ groups of their rigid actions. 

Recall the action canonical action $\widehat\al:\Ga^V\act\{\0,\1\}^\N$ for an FS group $\Ga^V$ satisfying assumptions (\ref{ass:1},~\ref{ass:2},~\ref{ass:3}) described in section \ref{sec:dynamics}. This action is rigid. 

For the usual action of Thompson's group $V\act\{\0,\1\}^\N$ it is standard that the group of germs $[V\act\{\0,\1\}^\N;p]\simeq\Z$ if $p$ has periodic tail and is trivial otherwise. By propositions \ref{prop:germ-pres} and \ref{prop:germ-not-singular} we deduce:
\begin{align*}
[G^V_t\act\{\0,\1\}^\N;p]&\simeq
\begin{cases}
\Gr\la a,b|a=b^{N_t}\ra&~~~~\textnormal{if $p$ has tail $\ov\1$,}\\
\Z&~~~~\textnormal{if $p$ has periodic tail not equal $\ov\1$,}\\
1&~~~~\textnormal{otherwise,}
\end{cases}\\
[J^V_s\act\{\0,\1\}^\N;p]&\simeq
\begin{cases}
\Gr\la a,b|a^2=b^{N_s+1}\ra&\textnormal{if $p$ has tail $\ov\1$,}\\
\Z&\textnormal{if $p$ has periodic tail not equal $\ov\1$,}\\
1&\textnormal{otherwise.}
\end{cases}
\end{align*}
On one hand we have $\Gr\la a,b|a=b^{N_t}\ra\simeq\Z$ generated by $b$ and on the other hand we obtain a quotient $$\Gr\la a,b|a^2=b^{N_s+1}\ra\onto\Z/2\Z*\Z/(N_s+1)\Z$$ by imposing the relation $a^2=b^{N_s+1}=e$.
The tree $s$ was assumed non-trivial implying that $N_s+1\geqslant 3.$
This implies that $\Z/2\Z * \Z/(N_s+1)\Z$ is non-abelian (it is not even amenable).
Now, $D(G_t^V)$ has a rigid action whose full group is $G_t^V$.
The rigid action of $G_t^V$ is conjugate to the action $G_t^V\act\{\0,\1\}^\N$.
Hence, $D(G_t^V)$ remembers the list of germ groups for the action $G_t^V\act\{\0,\1\}^\N.$
The same apply to $J_s^V$.
Therefore, if $D(G_t^V)$ were isomorphic to $D(J_s^V)$, there would exist a homeomorphism 
$\psi$ of the Cantor space $\{\0,\1\}^\N$ so that $$[G_t^V\act\{\0,\1\}^\N;p]\simeq [J_s^V\act \{\0,\1\}^\N;\psi(p)]$$ which is a contradiction.
Similar arguments can be made for the $F$ and $T$-cases. 
Finally, observe that the germ groups of $G_t$ at $o$ and at $\omega$ are abelian. This implies that $D(G_t)\subset G^0$ and thus $D(G_t)=D(G^0)$ by Claim 3 of the proof of Theorem 3.5 in \cite{Brothier-Seelig24}. From there we deduce that $D(G_t)=D(G_t^0).$
However, the germ group of $J_t$ at $\omega$ is non-abelian. This implies that $D(J_t)\not\subset J_t^0$ and in particular $D(J_t)\neq D(J_t^0).$

\section*{Index of notations}
\begin{itemize}
    \item $F\subset T\subset V$ are the Thompson groups;
    \item $\cF$ a forest-skein (FS) category. $\cF^T,\cF^V$ the $T$ and $V$-type FS categories roughly obtained by adding cyclic permutation and all permutations to $\cF$, respectively.
    \item $\cT$ the set of trees inside an FS category $\cF$;
    \item $\cT_*$ the monoid of pointed trees inside an FS category $\cF$;
    \item $(t,j)\in \cT_*$ the tree $t$ with distinguished leaf $j$;
    \item $G^F\subset G^T\subset G^V$ are the $F,T,V$-type FS groups;
    \item $G$ often denotes $G^F$ an $F$-type FS group;
    \item $\alpha:G^V\act \fC$ is the canonical action;
    \item $o,\omega$ are the min and max of $\fC$, respectively;
    \item $\cT_o,\cT_\omega$ the submonoids of $\cT_*$ of pointed trees with distinguish leaf the first and last, respectively;
    \item $G^{0}$ is the subgroup of an $F$-type FS group $G$ equal to the $g\in G$ acting trivially in a neighbourhood of $o$ and $\omega$; 
    \item $i^f$ denotes the numbering of the first leaf of the $i$th tree of a forest $f$;
    \item $\fQ=\cT_*/\sim$ is the quotient space for the equivalence relation $\sim$ generated by $(t,i)\sim (tf,i^f)$;
    \item $\beta:\cT_*\act \fQ$ is the quasi-regular action;
    \item $\fQ^-,\fQ^+$ are copies of $\fQ$ that are two subspaces of $\fC$ analogous to the points that are eventually constant inside $\{\0,\1\}^\N$;
    \item $\{\0,\1\}^*$ the monoid of finite binary strings;
    \item $\{\0,\1\}^\N$ the Cantor space;
    \item $D(\Ga)$ is the derived subgroup of $\Ga$;
    \item $[[\Ga\act Z]]$ is the full group of the action $\Ga\act Z$: homeomorphisms $g$ of $Z$ for which we may cover $Z$ with finitely many open subsets on which $g$ coincide with an element of $\Ga$;
    \item $[\Ga\act Z;z]$ is the germ group of the action $\Ga\act Z$ at the point $z$;
    \item $\ga^X:G^X\act \fR^X$ rigid action, with $\fR^X$ the rigid space. We set $\gamma:=\gamma^F$ and $\fR:=\fR^F$;
\end{itemize}

\newcommand{\etalchar}[1]{$^{#1}$}

\end{document}